\documentclass[arxiv,reqno]{ml-gen}
\renewcommand{\marginpar}[1]{\relax}

\usepackage{amsmath} 
\usepackage{amssymb,amsfonts,mathrsfs,amscd,upref}
\usepackage[symbol]{footmisc}
\usepackage{perpage}
\MakePerPage[5]{footnote}
\usepackage[colorlinks=true,linkcolor=blue,citecolor=blue,backref=page]{hyperref}
\usepackage{math60,mlthm}
\usepackage{local}

\numberwithin{equation}{section}

\usepackage{euler}

\begin{document}
\title[Classification of traces on classical pseudodifferential operators]
{Classification of traces and hypertraces on spaces of classical
pseudodifferential operators}

\author{Matthias Lesch}
\address{Mathematisches Institut,
Universit\"at Bonn,
Endenicher Allee 60,
53115 Bonn,
Germany}

\email{ml@matthiaslesch.de, lesch@math.uni-bonn.de}
\urladdr{www.matthiaslesch.de, www.math.uni-bonn.de/people/lesch}

\thanks{The first author was partially supported by the 
        Hausdorff Center for Mathematics and the second 
        author was supported by the Max-Planck-Institut f\"ur Mathematik.
}

\author{Carolina Neira Jim{\'e}nez}
\address{Fakult\"at f\"ur Mathematik, 
Universit\"at Regensburg\\
Universit\"atsstrasse 31\\
93053 Regensburg, Germany}
\email{Carolina.Neira-Jimenez@mathematik.uni-regensburg.de}


\begin{abstract}
Let $M$ be a closed manifold and let $\CL^\bullet(M)$ be the algebra of
classical pseudodifferential operators. The aim of this note is to
classify trace functionals on the subspaces $\CL^a(M)\subset \CL^\bullet(M)$ 
of operators of order $a$. $\CL^a(M)$ is a $\CL^0(M)$--module for any real
$a$; it is an algebra only if $a$ is a non--positive integer. Therefore, 
it turns out to be useful to introduce the notions of pretrace and hypertrace. 
Our main result gives a complete classification of pre-- and hypertraces on
$\CL^a(M)$ for any $a\in\R$, as well as the traces on $\CL^a(M)$ for
$a\in\Z, a\le 0$. We also extend these results to classical pseudodifferential 
operators acting on sections of a vector bundle.

As a byproduct we give a new proof of the
well--known uniqueness results for the Guillemin--Wodzicki residue trace
and for the Kontsevich--Vishik canonical trace.
The novelty of our approach lies in the calculation of
the cohomology groups of homogeneous and $\log$--polyhomogeneous differential forms
on a symplectic cone. This allows to give an extremely simple proof of a
generalization of a Theorem of Guillemin about the representation of
homogeneous functions as sums of Poisson brackets.

\end{abstract}

\maketitle
\tableofcontents

\section{Introduction and formulation of the result}  
Let $M$ be a smooth closed connected riemannian manifold
of dimension $n>1$.\footnote{The case $n=1$ has some peculiarities due to the
non--connectedness of the cosphere bundle of $S^1$. As a consequence many
results need to be slightly modified in the case $n=1$ (see Remark
\ref{rem:OneDim}). These modifications
are more annoying than difficult and for the sake of a clean exposition they are
left to the reader.} We denote by $\CLa$ the space of classical 
pseudodifferential operators of order $a\in\R$ on $M$. There is a little
subtlety here which we need to clarify to avoid possible confusions:
by definition (cf.\ Eq.\ \eqref{eq:3.1}, \eqref{eq:classical} and Section
\plref{sectionclassicalpdos}) a classical pseudodifferential operator of order
$a$ is 
also a classical pseudodifferential operator of order $a+k$ for any
non--negative integer $k$; this convention ensures, e.g., that $\CLa$ is a vector
space and that $\CLa$ is a subspace of $\CL^{a+1}(M)$. 
However, for non--integral $r\ge 0$ the space $\CLa$ is $\emph{not}$
contained in $\CL^{a+r}(M)$. In fact, it is not hard to see that
for such $r$ one has $\CLa\cap \CL^{a+r}(M)=\CL^{-\infty}(M)$,
the latter being the space of smoothing operators.

It is well--known that the residue trace $\Res$, 
which was discovered independently by \textsc{Guillemin} \cite{Gui:NPW}
and \textsc{Wodzicki} \cite{Wod:NRF}, is up to normalization the unique
trace on the algebra $\CL^\Z(M)$ of integer order classical pseudodifferential
operators (\cite{Wod:NRF}, \textsc{Brylinski} and \textsc{Getzler} \cite{BryGet:HAPDO},
\textsc{Fedosov}, \textsc{Golse}, \textsc{Leichtnam}, and \textsc{Schrohe} \cite{FGLS:NRM},
\textsc{Lesch} \cite{Les:NRP}, for a complete account of traces and determinants of
pseudodifferential operators see the recent monograph by \textsc{Scott}
\cite{Sco:TDP}).  $\Res$ is non--trivial only on $\CL^k(M)$ for integers
$k\ge -n$, and it is complemented by the canonical trace, $\TR$, 
of \textsc{Kontsevich} and \textsc{Vishik} \cite{KonVis:GDE}. The latter is defined on
operators of real order $a\not=-n, -n+1,\ldots$, it extends the Hilbert space
trace on smoothing operators and it vanishes on commutators (for the precise
statement see Eq.\ \eqref{eq:CanonicalTrace} below).
By  \textsc{Maniccia, Schrohe,} and \textsc{Seiler} \cite{MSS:UKVT} it is the unique 
functional which is linear on its domain, 
has the trace property and coincides with the $L^2$--operator trace on trace--class operators.

A natural problem which arises is to characterize the traces on 
the spaces $\CL^a(M)$. First, one has to note that
$\CL^a(M)$ is always a $\CL^0(M)$--module; it is an algebra if and only if
$a\in\Z_{\le 0}=\bigl\{ 0, -1, -2,\ldots\bigr\}$. Let us call a functional
$\tau$ on $\CL^a(M)$ a \emph{hypertrace} (resp. \emph{pretrace})
if $\tau\bigl([A,B]\bigr)=0$ for $A\in\CL^0(M), B\in \CLa$ (resp. $A,B\in\CL^{a/2}(M)$),
see Definition \plref{def:PreHyperTrace}.

The above mentioned uniqueness results for $\Res$ and $\TR$ cannot extend to
$\CLa$ for a
simple reason: let $T$ be a distribution on the cosphere bundle $S^*M$ and
denote by $\sigma_a:\CLa\to \cinf{S^*M}$ the leading symbol. Due to
the multiplicativity of the leading symbol (Eq.\ \eqref{eq:SymbolMultiplicative})
the map $T\circ \sigma_a$ is a pretrace and a hypertrace on $\CLa$, and for 
$a\in\Z_{\le 0}$ it is a trace on $\CLa$.
$T\circ \sigma_a$ is called a leading symbol trace by \textsc{Paycha} and
\textsc{Rosenberg} \cite{PayRos:TCCLS}.

For $\CL^0(M)$ it was already proved by \textsc{Wodzicki} \cite{Wod:RCHS}
that any trace is a linear combination of $\Res$ and a leading symbol trace,
see also \textsc{Lescure} and \textsc{Paycha} \cite{LescPay:UMDEP},  and
\textsc{Ponge} \cite{Pon:TPO}.

Before stating our generalization of this result we need to introduce some
more notation: Firstly, for $a\le 0$ we will always consider
$\CLa$ as a subspace of $\cB\bigl(L^2(M)\bigr)$, the bounded linear operators
acting on the Hilbert space $L^2(M)$ of square integrable functions with respect to the
volume measure induced by the riemannian metric. The symbol $\Tr$ will be
reserved for the operator trace on the Schatten ideal $\cB^1(L^2(M))$ of trace
class operators on $L^2(M)$. 

Secondly, for a 
linear functional $\tau:\CL^b(M)\to\C$ and $a\le b$ with $b-a\in\Z$
we will use the abbreviation $\tau_a:=\tau\restriction \CLa$.

Thirdly, we introduce a convenient notation which combines $\TR$ and $\Res$. 
Namely, fix a \emph{linear} functional  
$\widetilde{\Tr}:\CL^0(M)\to\C$ such that for
$a\in\Z_{<-n}=\bigl\{-n-1,-n-2,\ldots\bigr\}$ 
\[\Trt_a=\Trt\restriction\CLa=\Tr\restriction\CLa=\Tr_a,\]
and put 
\begin{equation}\label{eq:134a}  
     \TRb_a:= 
\begin{cases}
       \TR_a, & \textup{if }a\in\R\setminus \Z_{\ge -n},\\
       \Trt_a, & \textup{if } a\in \Z, -n\le a < \frac{-n+1}{2},\\
       \Res_a, & \textup{if } a\in \Z, \frac{-n+1}{2}\le a.
\end{cases}       
\end{equation}

In this note we will prove:

\begin{theorem}\label{t:A} Let $M$ be a closed connected riemannian manifold
 of dimension $n>1$. 

\textup{1. } Let $a\in\R$ and let $\tau$ be a hypertrace on $\CLa$. 
 Then there are uniquely determined $\gl\in\C$ and a  distribution
 $T\in\bigl(\cinf{S^*M}\bigr)^*$ such that
\begin{equation}\label{eq:I1}  
 \tau=T \circ \sigma_a +
     \begin{cases}
        \gl\, \TRb_{a},  & \textup{if } a\notin\Z_{>-n}, \\
        \gl\, \Res_{a},  & \textup{if } a\in\Z_{>-n}.
     \end{cases}
\end{equation}
 
\textup{2. } Let $a\in\Z_{\le 0},$ and denote by 
 \[
 \pi_a:\CL^a(M)\longrightarrow \CLa/\CL^{2a-1}(M)
 \]
the quotient map. Let $\tau:\CL^{a}(M)\to\C$ be a trace. 
Then there are uniquely determined $\gl\in\C$ and $T\in \bigl(\CLa/\CL^{2a-1}(M)\bigr)^*$ such that
\begin{equation}
 \tau= \gl\,\TRb_a + T\circ\pi_a.      
\end{equation}
\end{theorem}

This Theorem is a summary of Theorem \ref{t4}, Theorem \ref{t5} and Corollary \ref{cor:ClassTracesCla}
in the text. It extends to the vector bundle case. This requires even more notation and is 
therefore not reproduced here in the introduction. 
The interested reader is referred to Theorem \ref{t6} in Section
\ref{s:TracesVectBund}.

Let us briefly describe the main steps in the proof of Theorem
\plref{t:A}:

In order to classify (pre--, hyper--)traces on $\CLa$ it is natural
to ask for a representation of an operator $A\in\CLa$ as a sum
of commutators. Indeed, the uniqueness of the residue trace
$\Res$ as the unique trace on the algebra $\CL^\Z(M)$ (see the first
paragraph of this Section) essentially follows from the fact
that there exist $P_1,\ldots,P_N\in \CL^{1}(M)$, $Q\in\CL^{-n}(M)$,
such that for any $A\in \CL^{a}(M)$ 
there exist $Q_1,\ldots,Q_N\in \CL^{a}(M)$ 
and $R\in \CL^{-\infty}(M)$ such that 
\begin{equation} \label{eq:CommutatorRepresentation1}
 A=\sum_{j=1}^N[P_j,Q_j] + \Res(A)\, Q + R.
\end{equation}
This is due to \textsc{Wodzicki} \cite{Wod:LISA}, see also
\cite[Prop.\ 4.7 and Prop.\ 4.9]{Les:NRP}.

Since the operators $P_1,\ldots,P_N$ are of order $1$ they do not
belong to $\CLa$ except if $a\in\Z_{\ge 1}$. Hence to classify
pre-- and hypertraces on $\CLa$ we need to generalize
\eqref{eq:CommutatorRepresentation1} such that the order of the
$P_1,\ldots,P_N$
can be chosen to be an arbitrary real number $m$.

Indeed we will prove in Theorem \plref{t2} below that  for real numbers $m, a$ there 
exist $P_1,\ldots,P_N\in \CL^{m}(M)$, such that for any $A\in \CL^{a}(M)$ 
there exist $Q_1,\ldots,Q_N\in \CL^{a-m+1}(M)$ 
and $R\in \CL^{-\infty}(M)$ such that 
\begin{equation} \label{eq:CommutatorRepresentation2}
 A=\sum_{j=1}^N[P_j,Q_j] + \Res(A)\, Q + R. 
\end{equation}

From this representation and the well--known fact that the Hilbert space trace
is the unique trace on $\CL^{-\infty}(M)$ (\textsc{Guillemin} \cite[Thm.\ A.1]{Gui:RTFIO}, see
Theorem \plref{t:SmoothCommutator} below) one now deduces the first line
of \eqref{eq:I1} (Theorem \plref{t4}).

For the second line of \eqref{eq:I1} (Theorem \plref{t5}) one still applies
\eqref{eq:CommutatorRepresentation2} but then in addition one 
needs to show that if $a\in \Z_{>-n}$ and if $\tau$ is a hypertrace on $\CLa$ then
$\tau\restriction\CL^{-\infty}(M)=0$. This follows from a result of
\textsc{Ponge} (\cite[Prop.\ 4.2]{Pon:TPO}, see Proposition
\plref{p:SmoothCommutator1} below), for which we present
an alternative proof (Lemma \plref{l:PDOkernel}).

In Subsection \plref{ss:AAT} we present an alternative approach
which is independent of Ponge's result. For this alternative approach we
received considerable help from Sylvie Paycha.

For proving Theorem \plref{t:A} (2.) as well as for showing that every
pretrace is a hypertrace we use a nice algebraic Lemma
(Lemma \plref{l:commCl02aaa}) due to Sylvie Paycha. 
Subsection \plref{ss:AAT} as well as Lemma \plref{l:commCl02aaa} are included
here with her kind permission; her generosity is greatly appreciated.
We emphasize that Lemma \plref{l:commCl02aaa} and Subsection \plref{ss:AAT}
are not needed to prove the classification results about hypertraces contained 
in Theorem \plref{t4}, \plref{t5}, and Theorem \plref{t6}. 

As expected, \eqref{eq:CommutatorRepresentation2} is proved using the symbol
calculus for pseudodifferential operators. Recall that the leading
symbol, $\sigma_a(A)$, of $A\in\CLa$ is a smooth function on $T^*M\setminus M$
which is homogeneous of degree $a$. Now suppose we have $P\in\CL^m(M),
Q\in\CL^{a-m+1}(M)$. Then there is the well--known but crucial identity
\begin{equation} 
 \sigma_a([P,Q])=\frac 1i \{ \sigma_m(P),\sigma_{a-m+1}(Q)\}.
\end{equation}
Here, $\{\cdot,\cdot\}$ denotes the Poisson bracket of functions 
on $T^*M\setminus M$ with respect to the standard symplectic structure.
So Poisson brackets are the symbolic counterpart of commutators and
therefore to solve the original problem one has to analyze the space
spanned by Poisson brackets of homogeneous functions. 
This leads naturally to the \emph{symplectic residue} which is the symbolic
analogue of the residue trace. The theory of the symplectic residue was
developed independently by \textsc{Wodzicki} \cite[Sec.\ 1]{Wod:NRF} and
\textsc{Guillemin} \cite[Sec.\ 6]{Gui:NPW}. 

As in loc. cit. we work in the language of symplectic cones: $Y:=T^*M\setminus
M$ carries a natural free $\R_+^*$--action with quotient $S^*M$, the cosphere
bundle. For an arbitrary connected symplectic cone $Y$ denote by $\cP^a$ the
space of smooth functions which are homogeneous of degree $a$. If $Y$ is
of dimension $2n>2$ with compact base we prove in Theorem \plref{t:HomFuncPoissonBrack} below that
\begin{equation} \label{Intro:PlPmcases}
\begin{split}
\{\cP^{l},\cP^{m}\}&=\ker(\res_Y)\cap\cP^{l+m-1}\\
           &=     \begin{cases} \cP^{l+m-1}, & \textup{if }l+m\not=-n+1,\\ 
                                \ker(\res_Y)\cap\cP^{l+m-1}, & \textup{if }l+m=-n+1.
		                \end{cases}
\end{split}                  
\end{equation}
Here, $\res_Y$ denotes the symplectic residue (Definition \plref{defres1},
Section \plref{ss:SympRes}).

For $m=1$ this is \cite[Thm.\ 6.2]{Gui:NPW}, cf.\ also \cite[1.20]{Wod:NRF}.
The $m$ here corresponds to the $m$ in \eqref{eq:CommutatorRepresentation2}.
Hence, proving \eqref{Intro:PlPmcases} for arbitrary $m$ is crucial.
One could hope that the original method of \cite{Gui:NPW} can be adapted
to all $m$. As shown in \textsc{Neira Jim{\'e}nez} \cite[Sec.\ 1.4]{Nei:CCS} this indeed works
for $(l,m)\not=(0,0)$ but the method fails for the case $l=m=0$.
This was pointed out to the second author by Jean--Marie Lescure.

We therefore offer a completely new approach to the proof of
\eqref{Intro:PlPmcases} which is even more elementary than the proof
in \cite[Sec.\ 6]{Gui:NPW}; the latter uses the elliptic regularity Theorem.

Let us explain the basic idea of our approach: denote by $\go$ the symplectic
form on $Y$. Then $\go^n$ is a volume form. Furthermore, one has
the formula (1.2 in \cite{Wod:NRF})
\begin{equation}\label{Intro:I2} 
\{f,g\}\,\omega^n=d(g\,\iota_{X_f}\omega^n). 
\end{equation} 
Using this formula, an elementary calculation (see the proof of Theorem
\plref{t:HomFuncPoissonBrack}) shows that $f\in\cP^{l+m-1}$ is in
$\{\cP^l,\cP^m\}$ if and only if there is a \emph{homogeneous} differential
form $\beta$ (of homogeneity $n+l+m-1$) such that $f\go^n=d\beta$.

Thus the problem of proving \eqref{Intro:PlPmcases} is reduced to the
calculation of the $2n$-th de Rham cohomology of homogeneous differential
forms. It is no additional effort to calculate the whole homogeneous de Rham
cohomology of a cone: so let $Z$ be a smooth paracompact manifold and let
$\pi:Y\to Z$ be a $\R_+^*$ principal bundle over $Z$ (a cone). Denote by
$\Omega^p\cP^a(Y)$ the smooth $p$--forms which are homogeneous of degree
$a$ (see Subsection \plref{s:CohHomDifForms}). Then it is easy to see
that the exterior derivative preserves the homogeneity and hence
we can form the \emph{homogeneous de Rham cohomology groups} $H^p\cP^a(Y)$.
In Theorem \plref{t:200901071} we show that $H^p\cP^a(Y)$ vanishes for
$a\not=0$ and that for $a=0$ it is canonically isomorphic to $H^{p-1}(Z)\oplus
H^p(Z)$. In particular, for compact oriented $Z$ we find
that $H^{\dim Y}\cP^0(Y)$ is isomorphic to $\C$. The choice of a
homogeneous volume form for $Y$ (e.g.\ $\omega^n$ if $\go$ is the symplectic
form of a symplectic cone $Y$ of dimension $2n$) leads then to a concrete 
isomorphism $\res_Y:H^{\dim Y}\cP^0(Y)\to \C$. This is called the residue
of the cone. 

To finish the outline of the proof of Theorem \plref{t:A} let us explain
the connection between the residue of the cone $T^*M\setminus M$ (aka the
symplectic residue) and the residue trace: so let $M$ be compact connected of dimension $n>1$
and let $\go$ be the standard symplectic form on $T^*M\setminus M$. 
For $A\in\CLa$ the leading symbol $\sigma_a(A)$ is then
an element of $\cP^a(T^*M\setminus M)$. Furthermore, if $a\not=-n$
then the symplectic residue $\res_{\go}(\sigma_a(A))$ vanishes
and if $a=-n$ then $\res_{\go}(\sigma_a(A))$ is up to a normalization equal
to the residue trace $\Res(A)$ (cf.\ e.g.\ \cite[Prop.\ 4.5]{Les:NRP}).
This fact is used in the proof of Theorem \plref{t2} where
\eqref{eq:CommutatorRepresentation2} is deduced inductively from
\eqref{Intro:PlPmcases} using the symbol calculus.

There is another aspect which we would like to comment on. Namely,
it is interesting to note that $\Res$ and $\TR$ as well as the leading
symbol traces have precise analogues on the symbolic level.
This analogy is not only formal but is used in Subsection \plref{ss:AAT}. 

The basic idea is easy to explain, cf. also \cite[Sec.\ 4]{Les:PDO}: 
Let $U\subset \R^n$ be an open subset and let $A\in\CL^a(U)$ with complete
symbol $\sigma\in \CS^a(U\times\R^n)$ ($\CS^a$ denotes the space of classical
symbols of order $a$, see Subsection \plref{s:Symbols}). Then the Schwartz kernel
of $A$ is given by the oscillatory integral (cf.\ Eq.\ \eqref{eq:psido})
\begin{equation}\label{Intor:psido}%
 K_A(x,y)= \int_{\R^n} \, e^{i \langle x-y,\xi \rangle} \, \sigma(x,\xi) \, \dbar
 \xi, \qquad \dbar\xi:=(2\pi)^{-n} d\xi.
\end{equation}
To obtain a trace on $\CL^a(U)$ one hence has to regularize the integral
\begin{equation}
 \int_U K_A(x,x) dx = \int_U \int_{\R^n} \sigma(x,\xi)\, \dbar\xi dx.
\end{equation}
Only the inner integral is problematic and there are two natural regularizations
of the inner integral, the residue and the cut--off integral,
which then lead to $\Res$ and $\TR$ (cf.\ Subsection \ref{ss:SymRes}). Let us
ignore the $x$--dependence and consider the 
H\"ormander symbols $\CS^a(\R^n)\bigl(=\CS^a(\{0\}\times \R^n)\bigr)$.
This is the space of smooth functions $f$ on $\R^n$ such that
$f\sim\sum_{j=0}^\infty f_{a-j}$ with $f_{a-j}(\xi)$ positively homogeneous of order $a-j$
for $\xi$ large enough. 

In view of the fact that the symbolic analogue of commutators are Poisson brackets
and in view of the explanations after Eq.\ \eqref{Intro:I2} the analogue
of a hypertrace is then a linear functional $\tau:\CS^a(\R^n)\to \C$ such that $\tau(f)=0$ 
if the $n$--form $\ga=f d\xi_1\wedge\ldots d\xi_n$ is exact
within forms whose coefficients lie in $\CS^{a+1}(\R^n)$. Now for
$\ga$ to be exact in this sense it is equivalent that $f=\sum_{j=1}^n
\pl_{\xi_j}\sigma_j$ with $\sigma_j\in\CS^{a+1}(\R^n)$. This follows from
an elementary calculation, cf. the proof of Corollary \plref{cor:homres}.

In sum the analogue of a hypertrace is a linear function $\tau$ on
$\CS^a(\R^n)$ such that $\tau(\pl_{\xi_j} f)=0$ for
$j=1,\ldots,n$. Such functionals have been investigated by \textsc{Paycha}
\cite{Pay:NRC} and were partially classified (up to functionals on smoothing
symbols). As explained in e.g.\ \cite[Sec.\ 4.6.3]{Sco:TDP} studying these
functionals is one way to prove the existence of the residue trace; there is
another approach which makes more heavy use of heat trace asymptotics, cf.\
e.g.\ \cite[Sec.\ 4]{Les:NRP}.

Functionals with the \textquotedblleft Stokes property\textquotedblright,
$\tau(\partial_{\xi_j} f)=0$, can
most naturally be classified by looking at a certain variant of de Rham
cohomology. Namely, putting $T(fd\xi_1\wedge\ldots\wedge d\xi_n):=\tau(f)$
one obtains a linear function on the top degree de Rham cohomology of forms
in $\R^n$ whose coefficients lie in $\CS^a(\R^n)$. While the calculation
of this cohomology is possible, it will be postponed to a subsequent paper.
Rather it turns out that the homogeneous cohomology developed in Section
\ref{s:CohHomDifForms} plus a simple Lemma about Schwartz functions  (Lemma
\plref{symbolsumofder}) suffice to classify the functionals with the Stokes
property.

In Proposition \ref{p:SumDeriv} we completely
characterize the functionals on $\CS^a(\R^n)$ with the Stokes property
or equivalently when a function in $\CS^{a-1}(\R^n)$ can be written
as a sum of partial derivatives of functions in $\CS^a(\R^n)$.
This generalizes \cite[Prop.\ 2, Thm.\ 2]{Pay:NRC}.

The paper is organized as follows.
In Section \ref{s:CohHomDifForms} we study homogeneous differential forms on
cones and calculate their de Rham cohomology. As applications
we prove the aforementioned generalization of Guillemin's Theorem on
homogeneous functions and a characterization of functionals with the Stokes
property.

In Section \ref{s:POT} we first review some basic facts about
pseudodifferential operators and trace functionals. We introduce pretraces and hypertraces
and we give some examples. 
In Section \ref{s:OSC} we apply the results of Section \ref{s:CohHomDifForms} and provide
a result about the representation of a classical pseudodifferential 
operator as a sum of commutators. We use this result to give 
the classification of hypertraces and traces on $\CLa$ for different values of $a$. 
For the case of integral $a$ we give two proofs, one relying on a result due
to \textsc{Ponge} \cite{Pon:TPO} and a completely self--contained one in Subsection
\ref{ss:AAT}.

Finally, in Section \ref{s:TracesVectBund} we extend the results about tracial
functionals to operators acting on sections of vector bundles over the
manifold. The main result then is Theorem \plref{t6}.

\section*{Acknowledgments}

This paper exposes and extends some of the results of the Ph.D. Thesis
\cite{Nei:CCS} of the second author.  She would like to thank her adviser Matthias Lesch and
her co--adviser Sylvie Paycha for their guidance during this project, 
as well as the Max--Planck Institute f\"ur
Mathematik and the University of Bonn for their support and hospitality.
We acknowledge with gratitude the substantial help received from Sylvie
Paycha, in particular with Lemma \plref{l:commCl02aaa} 
and Subsection \plref{ss:AAT} which are included in this paper with
her kind permission.
Furthermore, we would like to thank Jean--Marie Lescure for pointing out an
error in an earlier draft. In fact this led us to develop the new approach
via homogeneous cohomology. 
Finally we thank the two anonymous referees for their detailed suggestions
for improvements. We think the paper has benefited considerably from those
remarks.
\section{Cohomology of homogeneous differential forms}\label{s:CohHomDifForms}        

In this section we calculate the de Rham cohomology of homogeneous differential
forms on cones. The theory is stunningly simple. Nevertheless as corollaries we obtain
generalizations of the results of \textsc{Guillemin} \cite{Gui:NPW} on the representation of homogeneous
functions on \emph{symplectic} cones as sums of Poisson brackets. Also our approach generalizes
the theory of homogeneous functions on $\R^n\setminus \{0\}$ in a straightforward way. Therefore,
we also obtain as a corollary the precise criterion when a homogeneous function can be written
as a sum of partial derivatives of homogeneous functions, cf.\ \cite{FGLS:NRM}, \cite{Les:NRP}. 
Finally, this criterion is generalized to classical symbol
functions, generalizing \cite[Prop.\ 2, Thm.\ 2]{Pay:NRC}.

\subsection{Homogeneous differential forms on cones }\label{s:HomDifForms}       

A \emph{cone} over a manifold $B$ is a principal bundle $\pi:Y\to B$ with
structure group $\rpluss$, the multiplicative group of positive real numbers.
Basic examples we have in mind are $\R^n\setminus \{0\}$
(cf.\ Examples \plref{sss:ExampleRn}, \plref{sss:ExampleRnpoly} below) and the cotangent bundle with the zero section
removed, $T^*M\setminus M$, of a compact connected manifold $M$; the latter
is even a symplectic cone and such cones are discussed in detail in Subsection
\plref{ss:SympCones}. In both cases the $\rpluss$ action is given by
multiplication.

Denote by $\varrho_\gl:Y\to Y$, the action of $\gl\in\rpluss$. 
Via $\Phi_t:=\varrho_{e^t}$ we obtain a one parameter group of diffeomorphisms of $Y$. 
Let $\cX\in\cinf{TY}$ be the infinitesimal generator of this group, 
which is sometimes called the \emph{Liouville vector field}.

A differential form $\go\in\Omega^p(Y)$ is called \emph{homogeneous} of degree
$a$ if $\varrho_\gl^*\go=\gl^a\go$ for all $\gl\in\rpluss$. The space of differential forms of form degree
$p$ and homogeneity $a$ is denoted by $\Omega^p\cP^a(Y)$. 
$\cP^a(Y):=\Omega^0\cP^a(Y)$ are the smooth functions on $Y$ which are homogeneous
of degree $a$.

We choose a function $r\in \cP^1(Y)$ which is everywhere positive and 
put $Z:=\bigsetdef{y\in Y}{r(y)=1}$. $\pi_{|Z}$ is a diffeomorphism
from $Z$ onto $B$ and $r$ induces a trivialization of $Y$ as follows:
\begin{equation}
         \Phi:Y \longrightarrow \rpluss\times Z,\quad
              y \ \mapsto (r(y), \varrho_{r(y)\ii}y).
\end{equation}
Note that 
\begin{equation}
      \Phi(\varrho_\gl(y))=(r(\varrho_\gl(y)), \varrho_{r(\varrho_\gl(y))\ii}\varrho_\gl(y))
                          = (\gl r(y),\varrho_{r(y)\ii}y). 
\end{equation}
Hence $\Phi$ intertwines the $\rpluss$ action on $Y$ and the natural $\rpluss$
action on the product $\rpluss\times Z$.
For convenience we will from now on work with the trivialized bundle $\rpluss\times Z$. 
The first coordinate will be called $r$, so the Liouville vector field is then
given by $\livf = r \frac{\pl}{\pl r}$. 

With the projection $\pi:\rpluss\times Z\to Z$, a differential form 
$\go\in\Omega^p\cP^a(\rpluss\times Z)$ can be written 
\begin{equation}\label{eq:125}
      \go=r^{a-1}dr\wedge\pi^*\tau+r^a\pi^*\eta
\end{equation}
with
\begin{equation}\label{eq:126}   
              \eta= i_Z^*\go \in\Omega^p(Z), \quad      \tau= i_Z^*(\iota_\livf \go)\in\Omega^{p-1}(Z), 
\end{equation}
where $i_Z:Z\hookrightarrow Y$ is the inclusion map and $\iota_\livf$ denotes
interior multiplication by the Liouville vector field $\livf$. We have
furthermore
\begin{equation}
 \label{eq:200901071}
        d\go=r^{a-1} dr\wedge(a\pi^*\eta-\pi^*d_Z\tau)+r^a\pi^*d_Z \eta\in\Omega^{p+1}\cP^a(\rpluss\times Z),
\end{equation}
so exterior derivation preserves the homogeneity degree.
Hence we can form the \emph{homogeneous de Rham cohomology groups} 
\begin{equation}\label{eq:100}  
 H^p\cP^a(Y):= \frac{\ker\big(d:\gO^p\cP^a(Y)\longrightarrow \gO^{p+1}\cP^a(Y)\big)}{ \im\big(d:\gO^{p-1}\cP^a(Y)\longrightarrow \gO^{p}\cP^a(Y)\big)}.   
\end{equation}
These cohomology groups can easily be calculated:

\begin{theorem}\label{t:200901071}
Let $Z$ be a smooth paracompact manifold, let $\pi:Y\to Z$ be a $\rpluss$
principal bundle over $Z$. 
\begin{enumerate}
 \item If $a\not=0$ then $H^p\cP^a(Y)=\{0\}$.
 \item If $a=0$ then the map 
  \begin{align*} \Psi:\Omega^\bullet\cP^0(Y)&\longrightarrow
  \Omega^{\bullet-1}(Z)\oplus \Omega^\bullet(Z)\\
  \go &\ \mapsto (\tau,\eta)=\bigl( i_Z^*(\iota_\livf\go),i_Z^* \go\bigr)
  \end{align*}
is an isomorphism of cochain complexes, hence it induces an isomorphism 
\begin{equation}\label{isomHcP}
   H^p\cP^0(Y) \cong H^{p-1}(Z)\oplus H^p(Z).
\end{equation}
In terms of the everywhere positive function $r\in\cP^1(Y)$
the inverse of $\Psi$ is given by 
$(\tau,\eta)\mapsto r^{-1}dr\wedge\pi^*\tau+\pi^*\eta.$
\end{enumerate}
\end{theorem}
\begin{proof}
(1) As before we work with the trivialized bundle $\rpluss\times Z$.
If $\go$ is closed then \eqref{eq:200901071} implies that
\begin{equation}
     d_Z\tau=a\eta,\quad d_Z\eta=0,
\end{equation}
and hence we obtain a form analogue of \emph{Euler's identity} 
(see Eq.\ \eqref{eq:eulerfunctns} below) 
\begin{equation}\label{eq:Euler}
      d\bigl( i_\livf \go\bigr)= d(r^a\pi^*\tau)
                       = a r^{a-1} dr\wedge\pi^*\tau +r^a \pi^*d_Z\tau
                       =a \go.
\end{equation}
Thus $\go$ is exact if $a\not=0$, explicitly 
\begin{equation}\label{eq:101}  
                 \go = \frac 1a  d\bigl( i_\livf \go\bigr).
\end{equation}

(2) Now let $a=0$ and consider $\omega\in\Omega^p\cP^0(\rpluss\times Z)$. 
Since $\varrho^*_{e^t}\go=\go$, we see that the Lie derivative
$\cL_\livf \go$ vanishes:
\begin{equation}\label{eq:127}  
               \cL_\livf \go = \frac{d}{dt}_{\bigl| t=0} \varrho^*_{e^t}\go
               =0,   
\end{equation}
and Cartan's magic formula
$d\iota_\livf+\iota_\livf d=\cL_\livf$ implies that $d\iota_\livf \go=-\iota_\livf
d\go$. Thus
\begin{equation}\label{eq:128}  
  d\bigl(i_Z^*(\iota_\livf \go),i_Z^*\go\bigr)  =  
  \bigl(-i_Z^*(\iota_\livf d\go),i_Z^*d\go\bigr),
\end{equation}
and hence the exterior derivative on $\gO^{\bullet-1}(Z)\oplus \gO^\bullet(Z)$
can be modified by a sign such that $d\Psi=\Psi d$. From \eqref{eq:125}
and \eqref{eq:126} it follows that $\Psi$ is bijective and that 
its inverse is given by 
$(\tau,\eta)\mapsto r^{-1}dr\wedge\pi^*\tau+\pi^*\eta.$
\end{proof}

\begin{remark}\label{rem:ResConstruction}
We comment on a special case of Theorem \plref{t:200901071} which combines the
constructions of the residue of a homogeneous function on $\R^n\setminus \{0\}$
(see the next Subsection) and of Guillemin's symplectic residue
(Subsection \ref{ss:SympCones}).

Let $\dim Y=n$ and suppose that $\go\in \gO^n\cP^a(Y)$ is a homogeneous
volume form. Then $i_Z^*(\iota_\livf \go)$ is a volume form on $Z$. In particular
$Z$ is orientable and we choose the orientation such that  $i_Z^*(\iota_\livf \go)$ 
is positively oriented. If additionally $Z$ is compact, then integration
yields an isomorphism $H^{n-1}(Z)\cong \C$.

For $f\in\cP^{-a}(Y)$ the closed form $f\go\in\Omega^{n}\cP^0(Y)$ defines a class
$[f\go]\in H^n\cP^0(Y)$ which under the isomorphism $\Psi$ of
Theorem \plref{t:200901071}
corresponds to the class $\bigl[i_Z^* (f\iota_\livf \go )\bigr] \in H^{n-1}(Z)$. 

\begin{dfn}\label{defres1} For $f\in\cP^{-a}(Y)$ we define the \emph{residue} 
with respect to the fixed volume form $\go\in \gO^n\cP^a(Y)$ to be 
the complex number corresponding to the class $[f\go]\in H^n\cP^0(Y)$ under the 
composition of the isomorphisms $H^n\cP^0(Y)\cong H^{n-1}(Z)\cong \C$:
\begin{equation}
\res_\go(f):=\int_{Z}i_Z^*\bigl( f\iota_{\cX}\go \bigr). 
\end{equation}
For $f\in \cP^b(Y), b\not=-a,$ we put $\res_\go(f)=0$.
\end{dfn}
We emphasize that the definition of $\res_\go$ depends on the choice
of the homogeneous volume form $\go$. The significance of Theorem
\plref{t:200901071} lies in the fact that $\res_\go(f)=0$ if and
only if there is a \emph{homogeneous} differential form $\beta$
such that $d\beta=f\go$.
\end{remark}

\subsubsection{Example} $Y=\R^n\setminus \{0\}\cong\rpluss\times S^{n-1},
B=Z=S^{n-1}$.\label{sss:ExampleRn}
We elaborate on this interesting special case.
Denote by $(\xi_1,\ldots,\xi_n)$ the coordinates on $\R^n\setminus \{0\}$ and put 
$\go:=d\xi_1\wedge\cdots\wedge d\xi_n\in\Omega^n\cP^n(\R^n\setminus \{0\})$. Then 
\begin{equation}\label{eq:129}  
\cX=\sum_{i=1}^n\xi_i\dfrac{\partial}{\partial\xi_i}, 
\qquad \iota_{\cX}\go=\sum_{i=1}^n(-1)^{i-1}\xi_i\,d\xi_1\wedge\cdots\wedge \widehat{d\xi_i}\wedge\cdots\wedge d\xi_n. 
\end{equation} 
The form $\iota_{\cX}\go$ is in $\Omega^{n-1}\cP^n(\R^n\setminus \{0\})$, and
$\pullbS(\iota_{\cX}\go)$ 
is the standard volume form on $S^{n-1}$. 
Moreover, by \eqref{eq:Euler} we have for $f\in\cP^{a}(\R^n\setminus \{0\})$
\begin{equation}\label{d(fiXomega)1}
 d(f\iota_{\cX}\go)=(a+n)f\,\go.
\end{equation} 
On the other hand by \eqref{eq:129} 
\begin{equation}
 d\bigl( f\iota_{\cX}\go \bigr)=\sum_{i=1}^n\pl_{\xi_i}(f\xi_i) \, d\xi_1\wedge\cdots\wedge d\xi_n
=\Bigl(\sum_{i=1}^n(\pl_{\xi_i}f) \xi_i + n\, f\Bigr)\go, \label{d(fiXomega)2}
\end{equation}
and thus we arrive at Euler's identity for homogeneous functions:
\begin{equation}\label{eq:eulerfunctns}
 \sum_{i=1}^n(\pl_{\xi_i}f)\xi_i=a\,f.
\end{equation}

\begin{cor}\label{cor:homres}
Let $\res_\go$ be the residue associated to 
$\go=d\xi_1\wedge\cdots\wedge d\xi_n\in \gO^n\cP^n(\R^n\setminus \{0\})$
according to Definition \plref{defres1}.
Then for a homogeneous function $f\in\cP^a(\R^n\setminus \{0\})$  the
following holds:
 \begin{enumerate}
  \item $\res_\go (\pl_{\xi_j} f)=0$.
  \item There exist $\sigma_j\in\cP^{a+1}(\R^n\setminus \{0\})$ 
   such that $f=\sum\limits_{j=1}^n\partial_{\xi_j}\sigma_j$ if and only if $\res_\go(f)=0$.
   Note that $\res_\go(f)\not=0$ only if $a=-n$.
 \end{enumerate}
\end{cor}
\begin{proof}
It follows from Theorem \plref{t:200901071} (cf.\ the remarks before Definition
\ref{defres1}) that for a function 
$g\in \cP^a(\R^n\setminus \{0\})$ the residue vanishes if and only if the class
$[ g\,\go]\in H^n\cP^{a+n}(\R^n\setminus \{0\})$ vanishes. 

To prove (1) we note that $(\pl_{\xi_j} f) d\xi_1\wedge\cdots\wedge\xi_n = d \eta$
with the form 
\[\eta= (-1)^{j-1} f d\xi_1\wedge\cdots\wedge \widehat{d\xi_j}\wedge\cdots\wedge d\xi_n 
    \in \gO^{n-1}\cP^{a+n-1}(\R^n\setminus \{0\})
\]    
and hence $\res_\go(\pl_{\xi_j} f) =0$.

(1) shows that for the $\sigma_j$ in (2) to exist it is necessary that $\res_\go(f)=0$. 
To prove sufficiency consider $f\in \cP^a(\R^n\setminus \{0\})$ with $\res_\go(f)=0$. 
Then there is $\eta\in \gO^{n-1}\cP^{a+n}(\R^n\setminus \{0\})$ with 
$d\eta=f\go$. We write
\begin{equation}\label{eq:EtaForm}
     \eta=\sum_{j=1}^n (-1)^{j-1} \sigma_j \,
     d\xi_1\wedge\cdots\wedge \widehat{d\xi_j}\wedge\cdots\wedge d\xi_n 
    \end{equation}
with $\sigma_j\in\cP^{a+1}(\R^n\setminus \{0\})$. 
Then $f=\sum\limits_{j=1}^n \partial_{\xi_j}\sigma_j$.
\end{proof}



\subsection{Extension to $\log$--polyhomogeneous forms}             
\label{s:LogPolforms}

\newcommand{\homforms}[3]{\Omega^#1\cP^{#2,#3}(\rpluss\times Z)}

We generalize our previous considerations to $\log$--polyhomogeneous forms.

A $p$--form $\go\in\Omega^p(\rpluss\times Z)$ is called \emph{$\log$--polyhomogeneous}
of degree $(a,k)$ if 
\begin{equation}
      \go=\sum_{j=0}^k \go_j\;\log^j r,
\end{equation}
with $\go_j\in\Omega^p\cP^a(\rpluss\times Z)$, cf.\ \cite{Les:NRP}.
The set of all such forms is denoted by $\homforms{p}{a}{k}$.

The exterior derivative preserves the $(a,k)$--degree. More explicitly,
\begin{multline}
      d\Bigl(\bigl(r^{a-1} dr\wedge\pi^*\tau+ r^a\pi^*\eta\bigr)\; \log^j r \Bigr) \\
             = \bigl( r^{a-1} dr\wedge (a\pi^*\eta-\pi^*d_Z\tau)+r^a \pi^*d_Z\eta\bigr)\;\log^j r 
             + j r^{a-1} dr\wedge \pi^*\eta\; \log^{j-1} r.
\end{multline}

Hence analogously to Eq.\ \eqref{eq:100} we define the \emph{$\log$--homogeneous de Rham
cohomology groups}
\begin{equation}\label{eq:102}  
 H^p\cP^{a,k}(Y):= \frac{\ker\big(d:\gO^p\cP^{a,k}(Y)\longrightarrow \gO^{p+1}\cP^{a,k}(Y)\big)}{ \im\big(d:\gO^{p-1}\cP^{a,k}(Y)\longrightarrow \gO^{p}\cP^{a,k}(Y)\big)},
\end{equation}
for which we can prove the following analogue of Theorem \plref{t:200901071}:

\begin{theorem}\label{t:200901072}
Let $Z$ be a smooth paracompact manifold, let $\pi:Y\to Z$ be a $\rpluss$
principal bundle over $Z$. Let $r\in\cP^1(Y)$ be everywhere positive. 
\begin{enumerate}
 \item If $a\not=0$ then  $H^p\cP^{a,k}(Y)=\{0\}$.
 \item If $a=0$ then the map 
  \begin{align*} \Phi^{k}:\Omega^{\bullet-1}(Z)\oplus
  \Omega^\bullet(Z)&\longrightarrow \Omega^\bullet\cP^{0,k}(Y)\\
(\tau,\eta)&\ \mapsto r\ii dr\wedge (\pi^*\tau)  \log^k r\;+\pi^*\eta
\end{align*} 
induces an isomorphism
\begin{equation}
 H^p(\Phi^{k}):H^{p-1}(Z)\oplus H^p(Z)\cong H^p\cP^{0,k}(Y).
\end{equation} 
Let $I^{p,k}: H^p\cP^{0,k}(Y)\to H^{p-1}(Z)\oplus H^p(Z)$ be the inverse
of $H^p(\Phi^k)$. Then for a closed form
$\go=\sum\limits_{j=0}^k \go_j \log^j r\in \gO^p\cP^{0,k}(Y),
\go_j\in\gO^p\cP^0(Y),$
one has $I^{p,k}([\go])=\bigl([i_Z^*  (\iota_\livf \go_k)],[i_Z^* \go_0]\bigr)$.
\end{enumerate}
\end{theorem}
\begin{proof}
We consider a \emph{closed} form $\omega\in\homforms{p}{a}{k}$ and write
\begin{equation}
    \go=\go_k \;\log^k r+\chi
\end{equation}
with $\chi\in\homforms{p}{a}{k-1}$. Then
\begin{equation}
    0=d\go=(d\go_k)\;\log^k r+\text{ lower log degree},
\end{equation}
thus $\go_k$ is closed and Euler's identity \eqref{eq:Euler}
gives
\begin{equation}
\begin{split}
      d\bigl( \iota_\livf \go\bigr) &=d\bigl( \iota_\livf \go_k\; \log^k r \bigr) + \text{ lower log degree}\\
      &= a\, \go_k\;\log^k r  + \text{ lower log degree }\\
      &= a\, \go + \text{ lower log degree}.
\end{split}
\end{equation}
If $a\not=0$ then $\go$ is cohomologous to $\go-\frac 1a d\bigl( i_\livf \go\bigr)
\in \homforms{p}{a}{k-1}$. By induction and Theorem \ref{t:200901071} one then shows
that $\go$ is exact. 

Next let $a=0$ and consider
a form $\go\in\homforms{p}{0}{k}$:
\begin{equation}
     \go=\sum_{j=0}^k\bigl( r\ii dr\wedge\pi^*\tau_j +\pi^*\eta_j\bigr) \;\log^j r,
\end{equation}
\begin{equation}
\begin{split}
      d\go&=\sum_{j=0}^k \bigl( -r\ii dr\wedge\pi^*d_Z\tau_j+\pi^*d_Z \eta_j\bigr) \;\log^j r + j  r\ii dr\wedge\pi^*\eta_j\; \log^{j-1} r\\
       &= \bigl( -r\ii dr\wedge\pi^*d_Z\tau_k+\pi^*d_Z\eta_k\bigr)\;\log^k r\\
         &\quad +\sum_{j=0}^{k-1}\bigl(r\ii dr\wedge((j+1)\pi^*\eta_{j+1}- \pi^*d_Z\tau_j) +\pi^*d_Z\eta_j\bigr)\;\log^j r.
\end{split}
\end{equation}
Thus $d\go=0$ if and only if
\begin{equation}
\begin{split}
        &d_Z \tau_k=0,\quad d_Z\eta_k=0,\\
        &d_Z \eta_j=0,\quad d_Z \tau_j=(j+1)\eta_{j+1},\quad j=0,...,k-1. 
\end{split}
\end{equation}
This implies that $H^p(\Phi^{k})$ and $I^{p,k}$ are well--defined and it is a
routine
matter to check that they are inverses of each other.
\end{proof}



\subsubsection{Example} $Y=\R^n\setminus \{0\},
B=Z=S^{n-1}$.\label{sss:ExampleRnpoly}
As in the homogeneous case we put:

\begin{dfn} Let $f\in\cP^{-n,k}(\R^n\setminus \{0\})$. We define the \emph{residue} of $f$ to be the integral 
\begin{equation}
\res_{\go,k}(f):=\res_\go(f_k)=\int_{S^{n-1}}\pullbS \bigl(
f_k\,\iota_{\cX}\go\bigr), \quad \go=d\xi_1\wedge\cdots\wedge d\xi_n. 
\end{equation}
\end{dfn} 

Note that by Theorem \ref{t:200901072}, $H^n\cP^{0,k}(\R^n\setminus \{0\})\cong H^{n-1}(S^{n-1})\cong \C$,
and that $\res_{\go,k}(f)$ is the image in $\C$ of the class $[f\go]$ under this isomorphism.
Therefore exactly as Corollary \ref{cor:homres} one now proves:
\begin{cor}\label{cor:loghomres}
 For a $\log$--polyhomogeneous function $f\in\cP^{a,k}(\R^n\setminus \{0\})$
 the following holds:
 \begin{enumerate}
  \item $\res_{\go,k}(\pl_{\xi_j} f)=0$.
  \item There exist $\sigma_j\in\cP^{a+1,k}(\R^n\setminus \{0\})$ such that 
  $f=\sum\limits_{j=1}^n\partial_{\xi_j}\sigma_j$ if and only if $\res_{\go,k}(f)=0$.
  Note that $\res_{\go,k}(f)\not=0$ only if $a=-n$.
 \end{enumerate}
\end{cor}


\subsection{Homogeneous functions on symplectic cones}\label{ss:SympCones}                


In this section we give an explicit expression of a homogeneous function in terms of Poisson brackets.
This generalizes work of \textsc{Guillemin} \cite[Thm.\ 6.2]{Gui:NPW}.

To fix some notation and to fix some (sign) conventions let us briefly collect some basic
facts from symplectic geometry:

Let $Y$ be a symplectic manifold with symplectic form $\go$. The Hamiltonian vector field $X_f$ associated to $f\in \cinf{Y}$ is characterized by $\iota_{X_f}\omega=-df$. 
The Poisson bracket of two functions $f,g\in \cinf{Y}$ is defined by 
\[\{f,g\}:=\omega(X_f,X_g).\] 
If $X_1$ and $X_2$ are Hamiltonian vector fields, then $[X_1,X_2]$ is also a Hamiltonian vector field with 
Hamiltonian function 
$\omega(X_1,X_2)$ (see Def.\ 18.5 in \textsc{Cannas da Silva} \cite{CdS:SG}): 
\begin{equation*}\iota_{[X_1,X_2]}\omega=\iota_{X_{\omega(X_1,X_2)}}\omega,\end{equation*} 
hence 
\begin{equation}\label{vectfieldcomm}X_{\{f,g\}}=X_{\omega(X_f,X_g)}=[X_f,X_g],\end{equation} 
and $(\cinf{Y},\{,\})$ is a Poisson algebra. 

\begin{prop}[1.2 in \cite{Wod:NRF}]\label{p:Wod} The Poisson bracket of 
two functions $f,g\in \cinf{Y}$ satisfies: 
\begin{equation}\label{I2} 
\{f,g\}\,\omega^n=n\,df\wedge dg\wedge\omega^{n-1}=d(g\,\iota_{X_f}\omega^n). 
\end{equation} 
\end{prop} 

Let $Y$ be a \emph{symplectic cone}, i.e.\ a cone $\pi:Y\to Z$ with a symplectic form $\go\in \gO^2\cP^1(Y)$. 
We assume furthermore that $Z$ is compact and connected; of course, $Y$ is then connected, too. 
The main example we have in mind is the cotangent bundle 
with the zero section removed, $T^*M\setminus{M}$, of a compact connected
manifold $M$ of dimension $\dim M > 1$,
with its standard symplectic structure.
The base manifold $Z$ is then the cosphere bundle $S^*M$. In the case $M=S^1$ (the only compact connected one--dimensional
manifold!), each of the two connected components of $T^*S^1\setminus S^1$ is a symplectic cone
over $S^1$.




\subsubsection{The symplectic residue}\label{ss:SympRes}

Let $\dim Y=:2n$, so $\go^n\in\Omega^{2n}\cP^n(Y)$ is a homogeneous volume form on $Y$.
We can apply Definition \ref{defres1} and define the \emph{symplectic residue} of a 
function $f\in\cP^a(Y)$ to be the residue
with respect to the volume form $\go^n$. That is 
\begin{equation}
\res_Y(f):=\res_{\go^n}(f)=\begin{cases} \int_Z i_Z^* (f \iota_\livf \go^n),  & \text{if }a=-n,\\
                              0,                           & \text{if }a\not=-n.
			      \end{cases}
\end{equation}
Recall that the definition of $\res_Y$ depends on the choice of the
homogeneous volume form $\go^n$. Furthermore, recall that by Theorem
\plref{t:200901071}
$\res_Y(f)=0$ if and only if there is a form
$\beta\in \gO^{2n-1}\cP^{a+n}(Y)$
such that $d\beta=f\go^n$. 

We note in passing that 
the form $\alpha:=\iota_{\cX}\omega$ is in $\Omega^1\cP^1(Y)$ and by Euler's identity for forms
Eq.\ \eqref{eq:Euler}, it satisfies $\omega=d\alpha$.
Our definition of the symplectic residue differs from the original one by
\textsc{Guillemin} \cite{Gui:NPW}
by a factor.

\subsubsection{Homogeneous functions in terms of Poisson brackets}
\label{sss:HomFuncPoisson}

Now we prove the following generalization of \cite[Thm.\ 6.2]{Gui:NPW}. 

In the following we will for brevity write $\cP^a$ instead of $\cP^a(Y).$
\begin{theorem}\label{t:HomFuncPoissonBrack}
 Let $Y$ be a connected symplectic cone of dimension $2n>2$ with compact base.
Then for any real numbers $l, m$ the following holds
\begin{equation} \label{PlPmcases}
\begin{split}
\{\cP^{l},\cP^{m}\}&=\ker(\res_Y)\cap\cP^{l+m-1}\\
           &=     \begin{cases} \cP^{l+m-1}, & \textup{if }l+m\not=-n+1,\\ 
                                \ker(\res_Y)\cap\cP^{l+m-1}, & \textup{if }l+m=-n+1.
		                \end{cases}
\end{split}                  
\end{equation}
\end{theorem}
\begin{remark}
The proof we present is based on the homogeneous cohomology developed in
Subsection \plref{s:HomDifForms}, in particular Theorem \plref{t:200901071}. 
While \cite{Gui:NPW} uses the elliptic regularity Theorem, 
our Theorem \plref{t:200901071} is completely elementary. 
More importantly our result is more general than loc.\ cit.\ where $m=1$ is
assumed. The technique of \cite[Sec.\ 6]{Gui:NPW} can be applied to prove
Theorem \plref{t:HomFuncPoissonBrack} for $(l,m)\not=(0,0)$ but the method 
fails\footnote{The second author would like to thank Jean--Marie Lescure for pointing this out to her.} 
 for the case $l=m=0$; for details see \cite[Sec.\ 1.4]{Nei:CCS}.
\end{remark}
\begin{proof} We first note that Proposition \ref{p:Wod} implies that $\{\cP^l,\cP^m\}\subset \cP^{l+m-1}$.
Furthermore, by loc.\ cit.\ we have $\{f,g\}\, \go^n = d(g\,\iota_{X_f}\omega^n)$, and if 
$f\in\cP^l, g\in \cP^m$ then $g\,\iota_{X_f}\omega^n \in \gO^{2n-1}\cP^{l+m+n-1}$.
Thus the homogeneous cohomology class of $\{f,g\} \go^n$ vanishes and hence 
$\res_Y(\{f,g\})=0$. So $\{\cP^{l},\cP^{m}\}\subset \ker(\res_Y)$.

Conversely, let $f\in\cP^{l+m-1}$ be given with $\res_Y(f)=0$. Then by Theorem
\plref{t:200901071} (see also Definition \plref{defres1}), the homogeneous
cohomology class of $f\go^n\in\gO^{2n}\cP^{n+l+m-1}$ vanishes and hence there is
a $\beta\in\gO^{2n-1}\cP^{n+l+m-1}$ such that
\begin{equation}\label{eq:103}  
                f \go^n = d\beta.    
\end{equation}

\subsubsection*{1. $l\not=0$ or $m\not=0$} Since the claim is symmetric in $l$ and $m$ 
we may, without loss of generality, assume that $l\not=0$. 

Choose functions $g_1,\ldots,g_N \in \cP^{l}$ such that at every point $y$ of $Y$
their differentials $dg_1|_y,\ldots,dg_N|_y$ span the cotangent space $T^*_yY$. Let
$X_1,\ldots,X_N$ be the Hamiltonian vector fields of $g_1,\ldots,g_N$.
Since $\go^n$ is a volume form also $\iota_{X_1}\go^n|_y,\ldots, \iota_{X_N}\go^n|_y$
span $\Lambda^{2n-1}T^*_yY$.

Consequently, there are functions $f_1,\ldots, f_N\in \cinf{Y}$ such that
\begin{equation}\label{eq:104}  
 \beta = \sum_{j=1}^N f_j\, \iota_{X_j}\go^n.
\end{equation}
Since $\beta, X_j, \go^n$ are homogeneous it is clear that also $f_j$ can be chosen to
be homogeneous. Counting degrees then shows $f_j\in \cP^m$.
Thus by Proposition \ref{p:Wod}
\begin{equation}\label{eq:105}  
 \begin{split}
       f \,  \go^n & = d\beta = \sum_{j=1}^N d( f_j\, \iota_{X_j} \go^n ) \\
	       & = n \sum_{j=1}^N dg_j\wedge df_j \wedge \go^{n-1} = \sum_{j=1}^N \{g_j, f_j\}\, \go^n,
 \end{split}
\end{equation}
and hence $f=\sum\limits_{j=1}^N \{g_j,f_j\}\in\{\cP^{l},\cP^{m}\}$.

\subsubsection*{2. $l=m=0$} In this case $f\in\cP^{-1}$. By assumption, $n>1$ and thus by 
Eq.\ \eqref{eq:Euler}
\begin{equation}\label{eq:106}  
 f\, \go^n = \frac{1}{n-1} d ( f\iota_\livf \go^n)= \frac{n}{n-1} d (f \ga
 \wedge \go^{n-1}), \quad \ga=\iota_\livf\go.
\end{equation}
The $1$--form $f\ga$ is homogeneous of degree $0$ and since $\ga=\iota_\livf\go$, it is
the pullback of a $1$--form on $Z$. 

We now choose $g_1,\ldots,g_N\in\cP^0$ such that at every point $z$ of $Z$, their differentials
span the cotangent space $T^*_zZ$. Of course it is impossible to find homogeneous functions
of degree $0$ such that their differentials span $T^*_yY$ at every $y\in Y$.

Therefore there are functions $f_1,\ldots,f_N\in\cinf{Y}$ such that
\[f \alpha=\sum_{i=1}^N f_i \,dg_i.\]  
As before, we see that $f_i$ can be chosen such that $f_i\in\cP^0$. 
Moreover, continuing Eq.\ \eqref{eq:106} and again using Proposition
\ref{p:Wod}
\begin{align*}
 f\,\omega^n&=\frac{n}{n-1}d(f\alpha)\wedge\go^{n-1}=\frac{n}{n-1}d\left(\sum_{i=1}^N f_i \,dg_i\right)\wedge\omega^{n-1}\\
            &=\frac{1}{n-1} \sum_{i=1}^N \{ f_i,g_i\}\,\omega^{n},
\end{align*} 
and we reach the conclusion $f=\frac{1}{n-1} \sum\limits_{i=1}^N \{ f_i,g_i\}\in \{\cP^0,\cP^0\}$.
\end{proof}

\begin{remark}\label{rem:OneDim} If $n=1$, then $\{\cP^0,\cP^0\}=0$. Indeed, by Eq.\ \eqref{I2} with $n=1$, 
$\{f,g\}\,\omega=df\wedge dg$, so if $f,g\in\cP^0$ we have $\{f,g\}=0$. In this one--dimensional case,
there are two different symplectic residues ($\res^+$, $\res^-$), corresponding to each 
connected component of $T^*S^1\setminus S^1$; then, when $l\not=0$ or $m\not=0$
we can argue as in the corresponding part of the proof of Theorem \ref{t:HomFuncPoissonBrack}, to
conclude that 
\begin{equation} \label{PlPmcases-a}
\{\cP^{l},\cP^{m}\}=
\begin{cases} \cP^{l+m-1}, & \text{if }l+m\not=0,\\ 
                                \ker(\res^+)\cap\ker(\res^-)\cap\cP^{l+m-1}, & \text{if }l+m=0.
\end{cases}                 
\end{equation}

\end{remark}

\subsection{The residue of a classical symbol function}              

As an application of homogeneous cohomology we give a precise criterion
when a classical symbol function is a sum of partial derivatives. A more thorough
discussion of de Rham cohomology of forms whose coefficients are symbol functions
will be given in a subsequent publication.


\subsubsection{Classes of symbols}\label{s:Symbols}

Let $U\subset \R^n$ be an open subset.
We denote by $\sym^m(U\times \R^N)$, $m\in \R$, the space of symbols 
of H\"ormander type $(1,0)$ (\textsc{H\"ormander} \cite{Hor:FIOI}, 
\textsc{Shubin} \cite{Shu:POST}) and order at most $m$. 
More precisely, $\sym^m(U\times \R^N)$ consists of those 
$a\in \CC^\infty(U\times \R^N)$ such that for multi--indices 
$\alpha\in \Z_+^n,\gamma\in \Z_+^N$ and compact subsets $K\subset U$ 
we have an estimate
\begin{equation}\label{eq:3.1}
    \bigl|\partial_x^\alpha\partial_\xi^\gamma a(x,\xi)\bigr|
  \le C_{\alpha,\gamma,K} (1+|\xi|)^{m-|\gamma|}, \quad x\in K,\ \xi\in\R^N.
\end{equation}
The best constants in \eqref{eq:3.1} provide a set of 
semi--norms which endow
$\sym^\infty (U\times \R^N):=\bigcup_{m\in\R}\sym^m(U\times \R^N)$ with the structure of a
Fr{\'e}chet algebra. 
A symbol $a\in\sym^m(U\times \R^N)$ is called \emph{classical} if there are
$a_{m-j}\in \cinf{U\times\R^N}$ with
\begin{equation}\label{eq:classical}
   a_{m-j}(x,r\xi)=r^{m-j} a_{m-j}(x,\xi),\quad r\ge 1, |\xi|\ge 1,
\end{equation}
such that for $N\in\Z_+$
\begin{equation}\label{eq:classical-a}
  a-\sum_{j=0}^{N-1} a_{m-j}\in\sym^{m-N}(U\times \R^N).
\end{equation}
The latter property is usually abbreviated $a\sim\sum\limits_{j=0}^\infty a_{m-j}$.

Homogeneity and smoothness at $0$ contradict each other except for homogeneous polynomials. Our convention
is that symbols should always be smooth functions, thus
the $a_{m-j}$ are smooth everywhere but homogeneous only in the restricted sense of
Eq.\ \eqref{eq:classical}. The homogeneous extension of $a_{m-j}$ to $U\times \R^n\setminus \{0\}$
will also be needed: we put 
\begin{equation}\label{eq:HomExtension}
 a_{m-j}^h(x,\xi):=a_{m-j}(x,\xi/|\xi|)\; |\xi|^{m-j},\quad (x,\xi)\in U\times \R^n\setminus \{0\}.
\end{equation}
Furthermore, we denote by
$\sym^{-\infty}(U\times\R^n):=\bigcap_{a\in\R}\sym^{a}(U\times\R^n)$
the space of \emph{smoothing symbols}. 

$\CS^m(U\times \R^N)\subset \sym^m(U\times\R^N)$ denotes the space of classical symbols of order $m$. 
Let us repeat the warning from the first paragraph of the introduction: in
view of \eqref{eq:3.1} and \eqref{eq:classical} one has $\CS^m(U\times
\R^N)\subset \CS^{m+r}(U\times\R^N)$ if and only if $r$ is a non--negative
integer. For non--integral $r\ge 0$ one has $\CS^m(U\times\R^N)\cap
\CS^{m+r}(U\times \R^N)=\sym^{-\infty}(U\times\R^N).$

Note that $\sym^{-\infty}(U\times \R^N)=\CS^{-\infty}(U\times
\R^N)=\bigcap_{a\in\R}\CS^{a}(U\times\R^n)$.

For brevity we write $\Csym^{a}(\R^n)$ ($\sym^a(\R^n)$) instead of $\Csym^a(\{\textup{pt}\}\times \R^n)$
($\sym^a(\{\textup{pt}\}\times \R^n)$). Note that $\sym^{-\infty}(\R^n)=\sS(\R^n)$ is nothing
but the Schwartz space of rapidly decaying functions.


We will now discuss the analogue of Corollary \ref{cor:homres} for the space $\Csym^a(\R^n)$.
We start with smoothing symbols:

\begin{lemma}\label{symbolsumofder} Let $f\in\sS(\R^n)$ be a Schwartz function. 
Then there are functions $\sigma_j\in\Csym^{-n+1}(\R^n)$ such that 
$f=\sum\limits_{j=1}^n\partial_{\xi_j}\sigma_j$. 

One can choose the $\sigma_j$ to be Schwartz functions if and only if $\int_{\R^n} f=0$.
\end{lemma}
\begin{proof} We start with the first claim and note that if $n=1$ then the function 
$\sigma(\xi)=\int_{-\infty}^\xi f(t)\,dt$
is in $\Csym^0(\R)$ and $\partial_\xi\sigma=f$.

For general $n$ we infer from the standard proof of the Poincar\'e Lemma in $\R^n$ applied 
to the closed form $f\,d\xi_1\wedge\cdots\wedge d\xi_n$, that we can put
\[\sigma_j(\xi)=\int_0^1f(t\xi)\,\xi_j\,t^{n-1}\,dt.\] 
Indeed, 
\[\partial_{\xi_j}\sigma_j(\xi)=
  \int_0^1f(t\xi)\,t^{n-1}\,dt+\int_0^1\partial_{\xi_j}(f)(t\xi)\,\xi_j\,t^{n}\,dt,
\] 
thus 
\begin{align*}
 \sum_{j=1}^n\partial_{\xi_j}\sigma_j(\xi) &=\int_0^1f(t\xi)\;n \,t^{n-1}\,dt+
       \int_0^1\sum_{j=1}^n\partial_{\xi_j}(f)(t\xi)\,\xi_j\,t^{n}\,dt \\ 
    &=\int_0^1\partial_t\Big(f(t\xi)\,t^{n}\Big)\,dt  
     =f(\xi).
\end{align*}
It remains to show that $\sigma_j\in\Csym^{-n+1}(\R^n)$. 
The function $\sigma_j$ is certainly smooth. 
For $|\xi|\geq1$ we have by change of variables $r=t|\xi|$: 
\begin{align*}
   \sigma_j(\xi)&=\int_0^{|\xi|}f\Bigl(r\dfrac{\xi}{|\xi|}\Bigr)r^{n-1}\,dr\;|\xi|^{-n}\,\xi_j \\ 
                &=\int_0^{\infty}f\Bigl(r\dfrac{\xi}{|\xi|}\Bigr)r^{n-1}\,dr\,|\xi|^{-n}\,\xi_j   
                       -  \int_{|\xi|}^\infty f\Bigl(r\dfrac{\xi}{|\xi|}\Bigr)r^{n-1}\,dr\,|\xi|^{-n}\,\xi_j.
\end{align*} 
The first summand is homogeneous of degree $-n+1$ while the second summand satisfies the estimates of a 
Schwartz function at $\infty$ (it is not a Schwartz function since it is not smooth at $0$). 
Thus $\sigma_j\in\Csym^{-n+1}(\R^n)$ and its homogeneous expansion consists only of one term of homogeneity $-n+1$: 
\[\sigma_j(\xi)\sim\int_0^\infty f\Bigl(r\dfrac{\xi}{|\xi|}\Bigr)r^{n-1}\,dr\,|\xi|^{-n\,}\xi_j,
\] 
proving the first claim.

For the second claim the necessity of $\int_{\R^n} f=0$ is clear. In fact the proof
of the Poincar\'e Lemma with compact supports (\textsc{Bott} and \textsc{Tu}
\cite[Sec.\ I.4]{BotTu:DFA}) works verbatim
for the forms $\gO^\bullet \sS(\R^n)$ with coefficients in $\sS(\R^n)$. Thus
the closed $n$--form $f d\xi_1\wedge\cdots\wedge d\xi_n$ is exact in $\gO^\bullet \sS(\R^n)$ 
if and only if $\int_{\R^n} f=0$. If this is the case then
$f d\xi_1\wedge\cdots\wedge d\xi_n=d\eta$ with an $(n-1)$--form $\eta\in \gO^{n-1}\sS(\R^n)$.
Expanding $\eta$ as in \eqref{eq:EtaForm} we see that 
$f=\sum\limits_{j=1}^n \partial_{\xi_j}\sigma_j$ with Schwartz functions $\sigma_j$.
\end{proof}

\subsubsection{The residue and the regularized (cut--off) integral}\label{ss:SymRes}
We now extend the residue (Def.\ \ref{defres1}) from homogeneous functions to $\CS^a(\R^n)$:

Let $\sigma\in\CS^a(\R^n)$ have asymptotic expansion $\sigma\sim 
\sum\limits_{j=0}^\infty \sigma_{a-j}$, cf.\ Eq.\ \eqref{eq:classical} and
\eqref{eq:HomExtension}. Then $\sigma^h_{a-j}\in\cP^{a-j}(\R^n\setminus\{0\})$.
Put
\begin{equation}\label{eq:SymRes}
\begin{split}
 \res(\sigma)&:=\res_{\go}(\sigma_{-n}^h) 
        = \int_{S^{n-1}} \pullbS ( \sigma_{-n}^h)\, d \vol_{S^{n-1}}\\
        &= \int_{S^{n-1}} \pullbS ( \sigma_{-n}^h \iota_\livf \go ), \quad
        \go=d\xi_1\wedge\ldots\wedge d\xi_n.
\end{split}
\end{equation}
In other words the residue of $\sigma$ equals the residue of its homogeneous
component of homogeneity degree $-n$. Thus $\res(\sigma)\not=0$ only if
$a$ is an integer $\ge -n$. The functional $\res$ was studied 
in \cite{Pay:NRC}.

We also recall the \emph{regularized integral} or cut--off integral
$\reginttext: \Csym^a(\R^n)\longrightarrow \C$ (cf.\ e.g.\
\textsc{Lesch} \cite[Sec.\ 4.2]{Les:PDO}): If $f\in\Csym^a(\R^n)$ then the asymptotic expansion
$f\sim\sum\limits_{j=0}^\infty f_{a-j}$ implies that as $R\to\infty$ one has an asymptotic 
expansion
\begin{equation}
\int_{|\xi|\leq R}f(\xi)\,d\xi\underset{R\to\infty}{\sim}\sum_{\substack{j=0\\
a-j+n\not=0}}^\infty c_{a-j}\,R^{a-j+n}+\widetilde{c}\,R^0+\res(f)\log R.
\end{equation}
The regularized integral $\reginttext_{\R^n}f(\xi)\,d\xi$ is, by definition, the 
constant term in this asymptotic expansion, i.e.\ $\widetilde{c}$. It has the 
property that $\reginttext_{\R^n} \pl_{\xi_j} f \not =0$ only if $a$ is an integer $\ge -n+1$.

The following result generalizes \cite[Prop.\ 2, Thm.\ 2]{Pay:NRC}
where it was proved modulo smoothing symbols.

\begin{prop}\label{p:SumDeriv} \textup{1.} Let $a\in\Z$. For a symbol 
$f\in\CS^a(\R^n)$ there exist symbols  $\sigma_j\in\CS^{r(a)}(\R^n), r(a):=\max(a,-n)+1,$ 
such that $f=\sum\limits_{j=1}^n \pl_{\xi_j}\sigma_j$ if and only if $\res(f)=0$.

\textup{2.} Let $a\in\R\setminus\Z$. For a symbol $f\in\CS^a(\R^n)$ there 
exist symbols $\sigma_j\in\CS^{a+1}(\R^n)$ such that 
$f=\sum\limits_{j=1}^n \pl_{\xi_j}\sigma_j$ 
if and only if $\reginttext_{\R^n} f=0$.
\end{prop}
\begin{proof} 
1. We will repeatedly use that by construction the asymptotic relation 
 Eq.\ \eqref{eq:classical-a} may be differentiated, i.e.\ if $g\in\CS^a(\R^n)$
 with $g\sim \sum\limits_{l=0}^\infty g_{a-l}$ then 
 \[
 \pl_{\xi_j} g  \sim \sum\limits_{l=0}^\infty \pl_{\xi_j} g_{a-l}.
 \]

Now let $a\in\Z$ and $f\in\CS^a(\R^n)$ with $f\sim \suml_{l=0}^\infty f_{a-l}$.
If $f=\sum\limits_{j=1}^n \pl_{\xi_j} \tau_j$ with $\tau_j\in\CS^{r(a)}(\R^n)$
then certainly $f_{-n}^h=\suml_{j=1}^n \pl_{\xi_j}\tau_{j,-n+1}^h$
and hence $\res(f)=\res(f_{-n}^h)=0$ by Corollary \ref{cor:homres}.

Conversely, if $\res(f)=0$ then again by Corollary \ref{cor:homres} there
are $\tau_{j,a-l+1}^h\in\cP^{a-l+1}(\Rnnull)$ such that
$f_{a-l}^h=\suml_{j=1}^n \pl_{\xi_j} \tau_{j,a-l+1}^h.$

We fix a cut--off function $\chi\in\cinf{\R^n}$ such that
\begin{equation}\label{eq:CutOffFctn}
 \chi(\xi)= \begin{cases}1, & \text{if } |\xi|\ge 1/2, \\
                       0, & \text{if } |\xi|\le 1/4.
          \end{cases}
\end{equation}	  
Now asymptotic summation \cite[Prop.\ 3.5]{Shu:POST} guarantees the existence of $\tau_j\in
\CS^{a+1}(\R^n)$ such that $\tau_j\sim \suml_{l=0}^\infty \chi \tau_{j,a-l+1}^h$
and hence
\begin{align}
   &\sum_{j=1}^n \pl_{\xi_j}\tau_j 
   \sim \sum_{l=0}^\infty \sum_{j=1}^n
      \chi \plxj  \tau_{j,a-l+1}^h \sim \sum_{l=0}^\infty f_{a-l} \sim f,\\
      \intertext{thus}
   &f-\suml_{j=1}^n \plxj \tau_j=:g\in \sym^{-\infty}(\R^n)=\sS(\R^n).
      \label{eq:107}  
\end{align}
Applying Lemma \ref{symbolsumofder} to $g$ the case $a\in\Z$ is settled.

2. Let $a\not\in\Z$. It was remarked before Proposition \ref{p:SumDeriv} that the condition
$\reginttext_{\R^n} f=0$ is necessary. To prove sufficiency consider
$f\in\CS^a(\R^n)$ with $\reginttext_{\R^n} f =0$.  Since $a\not\in\Z$, $\res(f)=0$
trivially. Therefore, as before we arrive
at \eqref{eq:107} (this is the content of \cite[Prop.\ 2]{Pay:NRC}).
Still we have $\int_{\R^n} g=\reginttext_{\R^n}f -\suml_{j=1}^n \reginttext_{\R^n}
\plxj \tau_j=0.$ Now apply the second part of Lemma \ref{symbolsumofder}
to $g$ and the proof is complete.
\end{proof} 


\section{Pseudodifferential operators and tracial functionals}                    
\label{s:POT}

\subsection*{Standing assumptions} Unless otherwise said in the rest of the
paper $M$ will denote a smooth closed connected riemannian manifold of dimension $n$. 
The riemannian metric is chosen for convenience only to have an
$L^2$--structure at our disposal. One could avoid choosing a metric by
working with densities.

Given $b\in\R$, we use the notation $\Z_{\le b}:=\Z\cap(-\infty,b]$, $\Z_{>b}:=\Z\cap(b,+\infty)$.
\subsection{Classical pseudodifferential operators}\label{sectionclassicalpdos}                

We denote by $\pdo^\bullet(M)$ the algebra of pseudodifferential operators 
with complete symbols of H\"ormander type $(1,0)$ 
(\cite{Hor:FIOI}, \cite{Shu:POST}), 
see Subsection \ref{s:Symbols}. The subalgebra of classical
pseudodifferential operators is denoted by $\CL^\bullet(M)$.

Let $U\subset \R^n$ be an open subset. Recall that for a symbol 
$\sigma\in\sym^m(U\times\R^n)$, the canonical pseudodifferential operator 
associated to $\sigma$ is defined by
\begin{equation}\label{eq:psido}%
\begin{split}
 \Op(\sigma)\, u (x) &:= \int_{\R^n} \, e^{i \langle x,\xi \rangle} \, 
                     \sigma(x,\xi) \, \hat{u} (\xi ) \, \dbar \xi \\
                     &= \int_{\R^n}\int_U \, e^{i \langle x-y,\xi \rangle} \, 
                       \sigma(x,\xi) \, u(y)\, dy\, \dbar \xi,
\end{split}\qquad \dbar\xi:=(2\pi)^{-n} d\xi.
\end{equation}
For a manifold $M$, elements of $\pdo^\bullet(M)$ (resp.\ $\CL^\bullet(M)$)
can locally be written as $\Op(\sigma)$ with $\sigma\in \sym^\bullet(U\times
\R^n)$ (resp.\ $\CS^\bullet(U\times\R^n)$).

Recall that there is an exact sequence
\begin{equation}\label{eq:LeadingSymbol}
0\longrightarrow \CL^{m-1}(M)\longhookrightarrow \CL^m(M)\xrightarrow{\sigma_m}
\cP^m(T^*M\setminus M)\longrightarrow 0,
\end{equation}
where $\sigma_m(A)$ is the homogeneous leading symbol of $A\in\CL^m(M)$.
$\sigma_m$ has a (non--canonical) global right inverse $\Op$ which is obtained 
by patching together the locally defined maps in Eq.\ \eqref{eq:psido}.
$\sigma_m(A)$ is a homogeneous function on the symplectic cone 
$T^*M\setminus M$ (cf.\ Subsection \ref{ss:SympCones}). We will tacitly
identify $\cP^m(T^*M\setminus M)$ by restriction with $\cinf{S^*M}$.
Here, $S^*M$ is the cosphere bundle, i.e.\ the unit sphere bundle
$\subset T^*M$. 

Recall that the leading symbol map is multiplicative in the sense
that 
\begin{equation}\label{eq:SymbolMultiplicative}
 \sigma_{a+b}(A\circ B)=\sigma_a(A)\,\sigma_b(B)
\end{equation}
for $A\in\CL^a(M), B\in\CL^b(M)$. Furthermore, we record the
important formula
\begin{equation}\label{eq:108}  
 \sigma_{a+b-1}\bigl([A,B]\bigr) = \frac 1i \{ \sigma_a(A),\sigma_b(B)\},
\end{equation}
which is a consequence of the asymptotic formula for the \emph{complete}
symbol of a product, cf.\ e.g.\ \cite[Thm.\ 3.4]{Shu:POST}.

\subsection{Tracial functionals on subspaces of $\CL^\bullet(M)$} \label{ss:Traces} 

Let $a\in\R$. $\CLa$ is an algebra if and only if $a\in\Z_{\le 0}.$ In this case a linear
functional $\tau:\CLa\longrightarrow\C$ is a trace if and only if 
\begin{equation}\label{eq:109}  
     \tau\bigl([A,B]\bigr) =0,\quad \text{ for all }A,B\in\CLa.
\end{equation}
Therefore, in order to characterize traces on $\CLa$, one has to understand the 
space of commutators $[\CLa,\CLa]$.
Note that the commutator $[A,B]\in\CL^{2a}(M)$.
Here, in the situation of operators with scalar coefficients, one even has
$[A,B]\in\CL^{2a-1}(M)$. However, $AB$ and $BA$ are only in $\CL^{2a}(M)$
and that $[A,B]\in\CL^{2a-1}(M)$ is only due to the fact that the
leading symbols of $A$ and $B$ commute. If $A,B$ are pseudodifferential
operators acting on sections of a vector bundle (see Section
\ref{s:TracesVectBund}) then one can only conclude that $[A,B]$ is of
order $2a$. 

Conversely, if
$\tau:\CL^{2a}(M)\longrightarrow\C$ is a linear functional satisfying
Eq.\ \eqref{eq:109} then any linear extension $\tilde \tau$ of $\tau$
to $\CLa$ is a trace on $\CLa$.

$\CL^{2a}(M)$ is a subspace of $\CLa$ if and only if $a\in\Z_{\le 0}$.
However, for any $a\in\R$ it makes sense to consider linear functionals
on $\CL^{2a}(M)$ satisfying \eqref{eq:109}:
\begin{dfn}\label{def:PreHyperTrace}
 Let $b\in\R$ and let $\tau:\CLb\longrightarrow \C$ be a
 linear functional.

 \textup{1. } $\tau$ is called a \emph{pretrace} if 
 $\tau\bigl([A,B]\bigr)=0$ for all $A,B\in \CL^{b/2}(M)$.

 \textup{2. } $\tau$ is called a \emph{hypertrace} if
 $\tau\bigl([A,B]\bigr)=0$ for all $A\in \CL^0(M), B\in\CLb$.

If $\CLa\subset\CLb$ we sometimes use the abbreviation
$\tau_a:=\tau\restriction\CLa$. 
\end{dfn} 
\begin{remark}
 If $b\in\Z_{\le 0}$ then any hypertrace on $\CLb$ is a trace
 on $\CLb$ since $\CLb\subset\CL^0(M)$. The restriction of a trace on $\CLb$ to
 $\CL^{2b}(M)$ is obviously a pretrace. 
\end{remark}

Next we discuss the canonical (pre--, hyper--)traces which exist on $\CLa$ for
various $a$. 

\subsubsection{The $L^2$--trace}\label{ss:L2Trace}
A pseudodifferential operator $A$ of order $\ord(A)<-n=-\dim M$ is a trace--class
operator. The standard Hilbert space trace on operators acting on $L^2(M)$ is denoted
by $\Tr$. Note that
\begin{equation}
     \Tr(A)=\int_{M}K_A(x,x)\,d\vol(x), \label{TrL2kernel}
    \end{equation}
where $K_A$ is the Schwartz kernel of the operator $A$. 
If $K_A$ is supported in a coordinate chart $U$ where $A$ is given as
$\Op(\sigma)$ with $\sigma\in\CS^a(U\times \R^n)$ then by Eq.\ \eqref{eq:psido}
\begin{equation}\label{defL2tracesymbol}
 \Tr(A)=\int_{U}\int_{\R^n}\sigma(x,\xi)\, \dbar\xi\, dx.
\end{equation} 

Since for any trace--class operator $K$ in the Hilbert space $L^2(M)$ and
any bounded operator $T$ in $L^2(M)$ one has $\Tr(KT)=\Tr(TK)$ it follows
that $\Tr$ is a hypertrace on $\CLa$ for any real $a<-n$. Furthermore,
if $p,q\ge 1$ are real numbers such that $1/p+1/q=1$  and if $A\in
 \sL^p(L^2(M))$, the $p$--th Schatten ideal of operators in $L^2(M)$,
 and $B\in\sL^q(L^2(M))$ then also $\Tr(AB)=\Tr(BA)$. 
 From  $\CLa\subset\sL^p(L^2(M))$ for $a<-n/p$ it then follows that
 \begin{equation}\label{eq:TrCommIsZero}
  \Tr\bigl([A,B]\bigr)=0, \textup{ for } A\in\CLa, B\in\CLb \textup{ if }
  a+b<-n,
 \end{equation} 
in particular $\Tr_a=\Tr\restriction \CLa$ is a pretrace for any $a<-n$. 
In fact, Eq.\ \eqref{eq:TrCommIsZero} can be improved slightly:

\begin{lemma}\label{afirmleq-n} Let $A\in\CLa$, $B\in\CLb$ with $a+b<-n+1$.
 Then $[A,B]$ is of trace--class and $\Tr\bigl([A,B]\bigr)=0$. 
\end{lemma}
\begin{proof}
We follow Sect.\ 4 of \cite{Les:NRP}. Let $P\in \CL^1(M)$ be an elliptic 
pseudodifferential operator whose leading symbol is positive and let $A\in
\CLa, B\in\CLb$. 
We put 
\[\nabla_P^0(B):=B, \quad \nabla_P^{j+1}B:=[P,\nabla_P^jB],\] 
and by induction, for all $j\in\N$ we have 
\[\nabla_P^jB\in \CL^{b}(M).\] 
Then, for $N$ large enough one has 
\[
  e^{-tP}B=\sum_{j=0}^{N-1}\dfrac{(-t)^j}{j!}(\nabla_P^jB)e^{-tP}+R_N(t),
\]
where $R_N(t)$ is a smoothing operator such that 
$\Tr\bigl(A R_N(t)\bigr)=\Tr\bigl(R_N(t) A\bigr)=O(t)$ as $t\to 0^+$; therefore 
\begin{equation}\label{eq:132}  
\Tr\bigl([A,B]e^{-tP}\bigr)=-\sum_{j=1}^{N-1}\dfrac{(-t)^j}{j!}\Tr\bigl(A(\nabla_P^jB)e^{-tP}\bigr)+O(t),\quad
t\to 0^+.
\end{equation}
Invoking the short time heat kernel asymptotics, cf.\ e.g.\ 
\textsc{Grubb} and \textsc{Seeley} \cite{GruSee:WPP}, 
\begin{equation}\label{eq:130}  
\Tr\bigl(A(\nabla_P^jB)e^{-tP}\bigr)\sim_{t \to 0^+}
       \sum_{k=0}^\infty(c_k+d_k\log t)t^{k-a-b-n}+\sum_{k=0}^\infty e_kt^k
\end{equation}
we see that for $j\ge 1$, thanks to $j-a-b-n>0$,
\begin{equation}\label{eq:131}  
     \lim\limits_{t\to 0^+} t^j\, \Tr\bigl( A (\nabla_P^jB)e^{-tP}\bigr)=0.
\end{equation}
Since $[A,B]\in \CL^{a+b-1}(M)$ and $a+b-1<-n$ the operator $[A,B]$ is of
trace--class and from \eqref{eq:132}, \eqref{eq:130}, and \eqref{eq:131} we thus infer
\[
\Tr\bigl([A,B]\bigr)=\lim\limits_{t\to 0^+}\Tr\bigl([A,B]e^{-tP}\bigr) =0.
\qedhere
\]
\end{proof}

\subsubsection{The Kontsevich--Vishik canonical trace}

For non--integer $a$ there is a regularization procedure which allows to
extend the $L^2$--trace in a canonical way to $\CLa$ 
(see \cite{KonVis:GDE},  \cite{Les:NRP}, \cite[Sec.\ 4.3]{Les:PDO}).
In brief for $a\in\R\setminus \Z_{\ge -n}$ there is a canonical
linear functional, \emph{the Kontsevich--Vishik canonical trace}, $\TR:\CLa\to \C$ such that
\begin{equation}
  \label{eq:CanonicalTrace}
  \begin{split}
   &\TR_a=\TR\restriction\CLa =\Tr\restriction\CLa=\Tr_a, \text{  if } a<-n,\\
      &\TR\bigl([A,B]\bigr)=0, \text{ if } A\in\CLa, B\in \CLb, a+b\not\in\Z_{\ge -n+1}.
     \end{split}
    \end{equation}     
Usually, the second property is stated only for $a+b\not\in \Z$. However,
if $a+b<-n$ then $AB$ is of trace--class and $\TR(AB)=\Tr(AB)=\Tr(BA)=\TR(BA)$
follows from the theory of the trace in Schatten ideals (see an analogous
discussion in the previous Subsection). If only $a+b-1<-n$ then 
$[A,B]$ is still of trace--class and $\TR\bigl([A,B]\bigr)=\Tr\bigl([A,B]\bigr)=0$ follows from 
Lemma \ref{afirmleq-n}.

The properties \eqref{eq:CanonicalTrace} immediately imply that the
canonical trace $\TR$ is a hypertrace and a pretrace on $\CLa$ for
$a\in\R\setminus \Z_{\ge-n}$.

\subsubsection{The residue trace}   

The \emph{residue trace}, called by some authors the noncommutative residue, somehow
complements the canonical trace. 
In terms of the complete symbol, the residue trace of an operator $A\in\CL^\bullet(M)$ is 
given by (see \cite{Wod:NRF}):
\[\Res(A)=\dfrac{1}{(2\pi)^n}\res(\sigma(A))=\dfrac{1}{(2\pi)^n}\int_M\int_{S_x^{*}M}\sigma_{-n}(A)(x,\xi)\,\nu(\xi)\wedge dx,\] 
where $\nu(\xi)$ is a volume form on $S_x^{*}M$ and $\res$ is the symplectic
residue on $T^*M\setminus M$ (cf.\ Subsection \plref{ss:SympRes} 
and Section \plref{ss:SymRes}). 
$\Res$ is the unique trace on the whole algebra $\CL^\bullet(M)$ whenever $n>1$ 
(\cite{Wod:NRF}, \cite{BryGet:HAPDO}, \cite{FGLS:NRM}, \cite{Les:NRP}). 
By definition, this trace vanishes on trace--class 
pseudodifferential operators and non--integer order pseudodifferential operators.

The residue trace $\Res$ is a pretrace and a hypertrace
on $\CLa$ for all $a\in\R$. It is non--trivial, however, only if $a\in\Z_{\ge -n}$.

\section{Operators as sums of commutators}    
\label{s:OSC}

In order to classify traces and (pre--, hyper--)traces on $\CLa$ we first
study the representation of an operator as a sum of commutators. 

\subsection{Smoothing operators }             

The closure of the algebra $\CL^{-\infty}(M)$ of smoothing operators in 
$\cB(L^2(M))$ is the algebra of compact operators. The latter is known
to be simple. Indeed one has the following, which is in a sense an analogue of
the second part of Lemma \plref{symbolsumofder}:

\begin{theorem}[{\cite[Thm.\ A.1]{Gui:RTFIO}}]\label{t:SmoothCommutator} 
 Let $M$ be a closed manifold. Then for any $J\in\CL^{-\infty}(M)$
 with $\Tr(J)=1$ the following holds: for $R\in \CL^{-\infty}(M)$
 there exist smoothing operators $S_1,\ldots,S_{N},
 T_1,\ldots,T_{N}\in\CLsm$, such that  
\[R=\Tr(R)\; J+\sum_{j=1}^{N}[S_j,T_j].\] 
Briefly, we have an exact sequence
\begin{equation}\label{eq:Exact}
0\longrightarrow [\CLsm,\CLsm]        \longrightarrow
                 \CLsm           \xrightarrow{\Tr}   \C  \longrightarrow 0.
                \end{equation} 
\end{theorem}

Can we write $J$ as a sum of commutators of general pseudodifferential
operators? Since $\Res$ is up to constants the only trace on $\CL^\bullet(M)$
(for $M$ compact and connected of dimension $>1$) the answer is yes.
A more precise answer is the following: 

\begin{prop}[See Prop.\ 4.2 in \cite{Pon:TPO}]\label{p:SmoothCommutator1}
Let $M$ be a compact riemannian manifold of dimension $n>1$.
Then $\CL^{-\infty}(M)\subset [\CL^0(M),\CL^{-n+1}(M)]$.
\end{prop}

We present here a brief variant of the proof of Ponge; our proof is based on
\begin{lemma}\label{l:PDOkernel} Let $n\ge 2$.

\textup{1. } The operator $Q_j$ of convolution by the function
\[f_j(y):=\dfrac{y_j}{|y|^2}=\partial_{y_j}(\log|y|)\] 
is a classical pseudodifferential operator of order $-n+1$ on $\R^n$.

\textup{2. } For any smoothing operator $R\in \CL^{-\infty}(\R^n)$ there exist
$B_j\in \CL^{-n+1}(\R^n)$, $j=1,\ldots,n$, such that
$R=\sum\limits_{j=1}^n [\Op(x_j),B_j] $.
\end{lemma}
\begin{remark}\label{r:MultiplicationOp}
$\Op(x_j)$ is the pseudodifferential operator associated to the
symbol function $(x,\xi)\mapsto x_j$. Of course, this is nothing
but the operator of multiplication by the coordinate $x_j$.
Therefore, $\Op(x_j)$ commutes with multiplication operators, a fact
which will often be used below.
\end{remark}
\begin{proof} \textup{1. } We have 
$f_j \restriction \Rnnull\in\cP^{-1}(\Rnnull)$.
Since $f_j$ is locally integrable on $\R^n$, it defines a distribution in
$\sD'(\R^n)$ which is homogeneous of degree $-1$.  
Then by \textsc{H\"ormander}  \cite[Thm.\ 7.1.18 and 7.1.16]{Hor:ALPI}, $f_j\in\sS'(\R^n)$
and its Fourier transform $\hat{f_j}$ is a homogeneous distribution of 
degree $-n+1$ in $\R^n$ which is smooth in $\Rnnull$. 
With the cut--off function $\chi$ of Eq.\ \eqref{eq:CutOffFctn} we therefore have
$\chi \hat{f_j}\in\CS^{-n+1}(\R^n)$. Furthermore, $(1-\chi)$ is compactly
supported and thus $(1-\chi)=\hat{\psi}$ with $\psi\in\sS(\R^n)$.
For $u\in \cinfz{\R^n}$, the space of compactly supported smooth functions on
$\R^n$, we now have
\[
 Q_j u   = f_j*u = \Op(\chi \hat{f_j})u + (\psi *f_j)*u.
\]
Convolution by the Schwartz function 
$\psi *f_j$ is smoothing and thus $Q_j\in\CL^{-n+1}(\R^n)$.

\textup{2. } A smoothing operator $R$ has a smooth kernel $K_R(x,y)$, and therefore, 
$(x,y)\mapsto K_R(x,y)-K_R(x,x)$ is smooth and vanishes on the diagonal. 
It follows that there are smooth functions $K_1,\ldots,K_n$ such that 
\[K_R(x,y)=K_R(x,x)+\sum_{j=1}^n(x_j-y_j)K_j(x,y).\] 
Let $Q$ be the operator defined by the kernel $K_Q(x,y)=K_R(x,x),$
and let $R_j$ be the smoothing operators defined by the kernels $K_j(x,y)$, then 
\[R=Q+\sum_{j=1}^n[\Op(x_j),R_j].\] 
Let $H_j$ be the operator with kernel $(x,y)\mapsto f_j(x-y) K_R(x,x).$
$H_j$ is $Q_j$ followed by multiplication by the smooth function $x\mapsto
K_R(x,x)$ and is therefore, by the proved part 1., a classical
pseudodifferential operator of order $-n+1$. Since 
\begin{align*}
 \sum_{j=1}^n(x_j-y_j) f_j(x-y)K_R(x,x)
   &=\sum_{j=1}^n\dfrac{(x_j-y_j)^2}{|x-y|^2}K_R(x,x) \\ 
   &=K_R(x,x) =K_Q(x,y),
\end{align*} 
it follows that $Q=\sum\limits_{j=1}^n[\Op(x_j),H_j].$
The result of the lemma follows with $B_j:=R_j+H_j\in \CL^{-n+1}(\R^n)$.
\end{proof}

\begin{proof}[Proof of Proposition \plref{p:SmoothCommutator1}]
Let $U\subseteq\R^n$ be an open set and let $R\in \CL^{-\infty}_{\comp}(U)$ 
be a smoothing operator with compactly supported Schwartz kernel 
$K_R\in \cinfz{U\times U}$. Let $\psi\in \cinfz{U}$ be such that 
$\psi(x)\psi(y)=1$ in a neighborhood of the support of the kernel of $R$, 
then $\psi R\psi=R$.

By Lemma \ref{l:PDOkernel} there exist $P_i\in \CL^{-n+1}(U)$ such that 
$R=\sum\limits_{i=1}^n[\Op(x_i),P_i]$. 
Let $\chi\in \cinfz{U}$ be such that $\chi=1$ in a neighborhood of
$\supp(\psi)$. Then we have 
\[\psi[\Op(x_i),P_i]\psi=
   \Op(x_i)\chi\psi P_i\psi-\psi P_i\psi\Op(x_i)\chi=[\Op(x_i)\chi,\psi P_i\psi],\] 
thus
\begin{equation}\label{smoothingU}
 R=\sum_{i=1}^n[\Op(x_i\chi),\psi P_i\psi].
\end{equation} 
Note that $x_i\chi\in \cinfz{U}$ and $\psi P_i\psi
\in\CL^{-n+1}_{\comp}(U)$.

Now let $\{\varphi_j\}\subset \cinf{M}$ be a partition of unity subordinate 
to a finite open covering $\{U_j\}$ of $M$ by coordinate charts. 
Furthermore, choose $\psi_j\in \cinfz{U_j}$ such that $\psi_j=1$ in
a neighborhood of $\supp(\varphi_j)$. 
Then for any $R\in \CL^{-\infty}(M)$ we have 
\begin{equation}\label{smoothingM}
 R=\sum_{j=1}^N\varphi_jR\psi_j+\sum_{j=1}^N\varphi_jR(1-\psi_j).
\end{equation} 
For each index $j$ the operator $\varphi_jR\psi_j$ belongs to 
$\CL^{-\infty}_{\comp}(U_j)$, so by the previous argument 
it can be written as a sum of commutators of the form \eqref{smoothingU}. 
Moreover, the operator 
$S:=\sum\limits_{j=1}^N\varphi_jR(1-\psi_j)$ is smoothing and its 
Schwartz kernel vanishes on the diagonal, so its 
trace vanishes and by Theorem \ref{t:SmoothCommutator} it can be written as a 
sum of commutators in 
$[\CL^{-\infty}(M),\CL^{-\infty}(M)]$. 
Hence $R$ belongs to the space $[\CL^0(M),\CL^{-n+1}(M)]$ as claimed.
\end{proof}

The degrees $0$ and $-n+1$ in the commutator $[\CL^0(M),\CL^{-n+1}(M)]$ in
Proposition \ref{p:SmoothCommutator1} can be traded against each other as the 
following simple but very useful Lemma due to Sylvie Paycha shows. This Lemma
is included with her kind permission.

\begin{lemma}\label{l:commCl02aaa}
For any $\ga,\gb\in\R$ 
\[[\CL^0(M),\CL^{\ga+\gb}(M)]\subset [\CL^{\ga}(M),\CL^{\gb}(M)],\] 
meaning that any commutator in $[\CL^0(M),\CL^{\ga+\gb}(M)]$ can be written as a 
sum of commutators in $[\CL^{\ga}(M),\CL^{\gb}(M)]$. 
\end{lemma}
\begin{proof}
 Let $A\in\CL^{0}(M)$, $B\in\CL^{\ga+\gb}(M)$. Fix a first order positive definite
 elliptic operator $\Lambda\in\CL^1(M)$. Then
$A \Lambda^\ga,\Lambda^{\ga}A, \Lambda^\ga \in\CL^\ga(M)$,
$B\Lambda^{-\ga}, \Lambda^{-\ga}B$, $AB \Lambda^{-\ga}, \Lambda^{-\ga}BA \in\CL^\gb(M)$.
Moreover,
\begin{align}
&[A \Lambda^{\ga}, \Lambda^{-\ga}B]= AB - \Lambda^{-\ga}BA\Lambda^{\ga},   \label{com1AB}\\
&[\Lambda^{\ga}A,  B\Lambda^{-\ga}]=\Lambda^{\ga}AB \Lambda^{-\ga}-BA,    \label{com2AB}\\
&[AB\Lambda^{-\ga}, \Lambda^{\ga}] =AB-\Lambda^{\ga}AB \Lambda^{-\ga},     \label{com3AB}\\
&[\Lambda^{-\ga}BA, \Lambda^{\ga}]= \Lambda^{-\ga}BA \Lambda^{\ga}-BA.     \label{com4AB}
\end{align} 
Adding up \eqref{com1AB}--\eqref{com4AB} yields twice the commutator $[A,B]$, 
whence $[A,B]\in [\CL^\ga(M),\CL^\gb(M)]$.
\end{proof}


\subsection{General classical pseudodifferential operators}\label{ss:GeneralPDO}               

We now combine the main result of Subsection \ref{ss:SympCones},
Theorem \ref{t:HomFuncPoissonBrack}, and the results of the previous
Subsection to obtain statements about general pseudodifferential
operators as sums of commutators. This improves, for classical
pseudodifferential operators, \cite[Prop.\ 4.7 and Prop.\ 4.9]{Les:NRP};
for such operators these results in fact go back to
\cite{Wod:LISA}. In \cite{Les:NRP} the more general class
of pseudodifferential operators with $\log$--polyhomogeneous symbol expansions
was considered.

\begin{theorem}\label{t2}
Let $M$ be a compact connected riemannian manifold of dimension $n>1$.
Fix $Q\in \CL^{-n}(M)$ with $\Res(Q)=1$. 
Then for any real numbers $m, a$ there 
exist $P_1,\ldots,P_N\in \CL^{m}(M)$, such that for any $A\in \CL^{a}(M)$ 
there exist $Q_1,\ldots,Q_N\in \CL^{a-m+1}(M)$ 
and $R\in \CL^{-\infty}(M)$ such that 
\begin{equation} 
 A=\sum_{j=1}^N[P_j,Q_j] + \Res(A)\, Q + R. 
\end{equation}
\end{theorem}
\begin{proof} 
 We follow the proof of \cite[Prop.\ 4.7]{Les:NRP}, where the case $m=1$ is 
 discussed, with a few modifications and improvements.

 First, replacing $A$ by $A-\Res(A)\, Q$ if necessary, we may,
 without loss of generality, assume that $\Res(A)=0$. 

 We choose $p_1,\ldots,p_N\in \cP^{m}(T^*M\setminus M)$ such that their
differentials span the cotangent bundle of $T^*M\setminus M$ at every point 
if $m\not=0$; if $m=0$ we choose the $p_j$ such that their differentials 
restricted to $S^*M$ span the cotangent bundle of 
$S^*M$ (cf.\ the proof of Theorem \ref{t:HomFuncPoissonBrack}). Choose
$P_j\in\CL^m(M)$ with leading symbols $p_j$.
Consider the leading symbol 
$\sigma_a(A)\in\cP^a(T^*M\setminus M)$ of $A$. 
Its symplectic residue is 0 if $a\not=-n$, and if $a=-n$ it is up to a normalization 
equal to $\Res(A)$ (cf.\ e.g.\ \cite[Prop.\ 4.5]{Les:NRP}), hence it is also $0$ in that case.

Then by Theorem \ref{t:HomFuncPoissonBrack} and its proof
there are $q_j^{(1)}\in \cP^{a-m+1}(T^*M\setminus M)$ such that
$\sigma_a(A)=\frac 1i \suml_{j=1}^N \{p_j,q_j^{(1)}\}$.
Thus choosing $Q_j^{(1)}\in\CL^{a-m+1}(M)$ with leading symbol
$q_j^{(1)}$ we find, see Eq.\ \eqref{eq:108},
\begin{equation*}
 A^{(1)}=A-\sum_{j=1}^N[P_j,Q_j^{(1)}]\in \CL^{a-1}(M). 
\end{equation*} 
We iterate the procedure: inductively, assume that we have operators
$Q_j^{(l)}\in\CL^{a-m+1}(M)$, $1\le l\le l_0$, such that
\[
 A^{(l)}=A-\sum_{j=1}^N[P_j,Q_j^{(l)}]\in \CL^{a-l}(M)
\]
and 
\begin{equation}\label{eq:133}  
     Q_j^{(l)}-Q_j^{(l+1)}\in\CL^{a-m+1-l}(M), \quad 1\le l\le l_0-1.
\end{equation} 
As for $A$ we then choose $B_j\in\CL^{a-m-l_0+1}(M)$ such that
\[
 A^{(l_0+1)}=A^{(l_0)}-\sum_{j=1}^N[P_j,B_j]\in \CL^{a-l_0-1}(M).
\]
Now put $Q_j^{(l_0+1)}=Q_j^{(l_0)}+B_j$. Then \eqref{eq:133}
holds for all $l$ and we can invoke the asymptotic summation principle
\cite[Prop.\ 3.5]{Shu:POST} and choose 
$Q_j\in\CL^{a-m+1}(M)$ such that for all $l\in\N$, $Q_j-Q_j^{(l)}\in \CL^{s-m+1-l}(M)$.
Then
\[
 A-\sum_{j=1}^N[P_j,Q_j]\in \CL^{-\infty}(M).\qedhere
\]
\end{proof}

Combining Theorem \ref{t2} and Lemma \ref{l:commCl02aaa} we find

\begin{theorem}\label{t3} Under the assumptions of Theorem \plref{t2} let
$a\in \Z, -n\le a< 0$. Then
\begin{align}
      \CLa &  =  [\CL^{(a+1)/2}(M),\CL^{(a+1)/2}(M)] \oplus \C \cdot
      Q,\label{eq:t31}\\
      &  =  [\CL^0(M),\CL^{a+1}(M)]    \oplus \C \cdot Q.\label{eq:t32}
\end{align}      
In other words for $A\in \CLa$
there exist operators $P_1,\ldots,P_N,Q_1,\ldots,Q_N\in \CL^{(a+1)/2}(M)$
resp. $P_1,\ldots,P_N\in\CL^0(M)$, $Q_1,\ldots,Q_N\in\CL^{a+1}(M)$ such that 
\begin{equation} 
 A=\sum_{j=1}^N[P_j,Q_j] + \Res(A)\, Q. 
\end{equation} 
\end{theorem}
\begin{proof} Apply Theorem \ref{t2} with $m=(a+1)/2$ (resp.\ $m=0$).
 This yields $P_1,\ldots,P_{N'}$ in $\CL^{(a+1)/2}(M)$ (resp.\ $\CL^0(M)$),
 $Q_1,\ldots,Q_{N'}\in \CL^{(a+1)/2}(M)$ (resp.\ $\CL^{a+1}(M)$) and
 $R\in \CLsm$ such that
 \[
 A=\sum_{j=1}^{N'}[P_j,Q_j] + \Res(A)\, Q +R. 
 \]
 By Proposition \ref{p:SmoothCommutator1} we have 
 \[
 \CLsm\subset [\CL^0(M),\CL^{-n+1}(M)]\subset [\CL^0(M),\CL^{a+1}(M)]
 \]
 and hence there are $P_{N'+1},\ldots,P_N\in\CL^0(M)$ and
 $Q_{N'+1},\ldots,Q_N\in \CL^{a+1}(M)$ such that
 $R=\suml_{j={N'+1}}^N [P_j,Q_j]$  proving Eq.\ \eqref{eq:t32}.

 To prove Eq.\ \eqref{eq:t31} we apply Lemma \ref{l:commCl02aaa}
 with $\ga=(a+1)/2, \gb=-n+1-\ga$. Then $\ga-\gb=a+n\in\Z_{\ge 0}$,
 hence $\CL^{\gb}(M)\subset \CL^\ga(M)$ and we find
 \[
 \begin{split}
 R  &\in \CLsm\subset [\CL^0(M),\CL^{-n+1}(M)]\\
    &\subset [\CL^\ga(M),\CL^\gb(M)] 
     \subset[\CL^{(a+1)/2}(M),\CL^{(a+1)/2}(M)].\qedhere
\end{split}
\]
\end{proof}


\subsection{Classification of traces on $\CL^{a}(M)$}\label{s:ClassTrac}               

We are now going to classify the pretraces and the hypertraces on $\CLa$ for 
all $a\in\R$, as well as the traces on $\CLa$ for $a\in\Z_{\leq0}$. The following 
definition will be convenient:

\begin{dfn}\label{def:TRRes}
Recall that for a linear functional $\tau:\CL^b(M)\to\C$ and 
$\CLa\subset\CLb$ we abbreviate $\tau_a:=\tau\restriction \CLa$.

We fix once and for all, a \emph{linear} functional  
$\widetilde{\Tr}:\CL^0(M)\to\C$
such that for $a\in\Z_{<-n}$ 
\[\Trt\restriction\CLa=\Tr\restriction\CLa,\]
cf.\ Definition \plref{def:PreHyperTrace}. Furthermore put 
\begin{equation}\label{eq:134}  
     \TRb_a:= 
\begin{cases}
       \TR_a, & \textup{if }a\in\R\setminus \Z_{\ge -n},\\
       \Trt_a, & \textup{if } a\in \Z, -n\le a < \frac{-n+1}{2},\\
       \Res_a, & \textup{if } a\in \Z, \frac{-n+1}{2}\le a.
\end{cases}       
\end{equation}
\end{dfn}

$\TRb_a$ conveniently combines the Kontsevich--Vishik trace and the residue
trace. The notation is slightly abusive since for $a,b\in \Z, a<(-n+1)/2\le b$
one has 
$\TRb_b\restriction \CL^{2a-1}(M)=\Res \restriction \CL^{2a-1}(M)=0
\not=\Tr\restriction \CL^{2a-1}(M)= \TRb_{2a-1}$. The disadvantages of this
notational conflict are outweighed by the convenience of having a common
notation for the Kontsevich--Vishik trace and the residue trace. This will
free us from repetitively having to make a distinction between the cases
$a\in\R\setminus \Z_{>-n}$ and $a\in\Z_{>-n}$.

We also emphasize that the choice of $\Trt$ is not canonical but certainly possible.

\begin{prop}\label{p:PreHyperTrace} Let $a\in\R$.
 
 \textup{1. }
 Any pretrace on $\CLa$ is a hypertrace on $\CLa$.
 
 \textup{2.}
 If $\tau$ is a hypertrace on $\CL^a(M)$ then there is a unique constant $\lambda\in\C$
 such that $\tau\restriction\CLsm=\lambda\, \Tr$.

 \textup{3.}
 If $a\in\Z_{\le 0}$ and $\tau$ is a trace on $\CL^{a}(M)$ then
 $\tau\restriction \CL^{2a}(M)$ is a pretrace (and hence a hypertrace). 
 Conversely, given a pretrace on $\CL^{2a}(M)$, any linear extension 
 $\widetilde \tau$ of $\tau$ to $\CL^{a}(M)$ is a trace.
 
 \textup{4.}
 For $a\in\Z_{\le 0}$, $\TRb_a$ is a trace on $\CLa$.
 For $a\in \R\setminus\bigl(\Z\cap [-n+1,-n/2]\bigr)$ it is a pretrace (and hence a hypertrace).
\end{prop}
\begin{proof} 1. follows from Lemma \ref{l:commCl02aaa}.

2. follows from Theorem \ref{t:SmoothCommutator}.

3. is obvious.

4. For $\frac{-n+1}{2}\le a$ the claim follows from the properties of the
residue trace.

Except for $a=-n$ the fact that $\TRb_a$ is a pretrace follows since
$\Res_a$ and $\TR_a$ are pretraces.

Next consider $a\in\R$, $a<\frac{-n+1}{2}$. Then for $A,B\in\CLa$ it follows from Lemma 
\ref{afirmleq-n} that $[A,B]\in\CL^{2a-1}(M)$ is of trace--class and that 
\[\TRb_{2a-1}\bigl([A,B]\bigr)=\Tr\bigl([A,B]\bigr)=0.\] 
This proves the remaining claims under 4.
\end{proof}

Thus to classify traces on $\CLa$ (for $a\in\Z_{\le0}$) it suffices to classify
pretraces on $\CL^{2a}(M)$. And to classify pretraces on $\CLb$ (for any $b\in\R$!)
it suffices to classify hypertraces on $\CL^{b}(M)$.

The following considerably improves a uniqueness result 
by \textsc{Maniccia}, \textsc{Schrohe}, and \textsc{Seiler} \cite{MSS:UKVT}.
    
\begin{theorem}\label{t4} Let $M$ be a closed connected riemannian manifold
 of dimension $n>1$, $a\in\R\setminus \Z_{> -n}$ and let $\tau$ be a hypertrace
 on $\CLa$. Then there are uniquely determined $\gl\in\C$ and a  distribution
 $T\in\bigl(\cinf{S^*M}\bigr)^*$
 such that $\tau=\gl\, \TRb_a + T\circ\, \sigma_a$.

 Consequently, a linear functional on $\CLa$ is a hypertrace if and only
 if it is a pretrace.
\end{theorem}
\begin{remark} 
Recall from Eq.\ \eqref{eq:LeadingSymbol} that $\sigma_a$ denotes the leading
symbol map.
Since the leading symbol is multiplicative (see Eq.\ \eqref{eq:SymbolMultiplicative})
it follows that for any $T\in\bigl(\cinf{S^*M}\bigr)^*$ the functional
$T\circ \sigma_a$ is a pretrace and a hypertrace on $\CLa$. Some authors (see \cite{PayRos:TCCLS}) call such traces
\emph{leading symbol traces}. 
\end{remark}
\begin{proof} We note
that if $\tau$ is a hypertrace on $\CLa$ then by Proposition
\ref{p:PreHyperTrace} (2.), there is a unique $\gl\in\C$ such that 
$\tau\restriction \CLsm=\gl \Tr.$


 We apply Theorem \ref{t2} with $m=0$. Then for $A\in\CL^{a-1}(M)$
 we find
\begin{equation}\label{eq:110}  
    A=\sum_{j=1}^N [P_j,Q_j]+R,
\end{equation}
 with $P_j\in\CL^0$, $Q_j\in\CLa$. Note that $\Res(A)=0$
 since $a-1\in\R\setminus \Z_{\geq-n}$.
 From Eq.\ \eqref{eq:110} we infer
 $\tau(A)=\tau(R)=\gl\Tr(R)=\gl\TR(R)=\gl \TR(A).$

 Thus we have $\tau\restriction\CL^{a-1}(M)=\gl\,
 \TR\restriction\CL^{a-1}(M)=\gl\,  \TR_{a-1}=\gl\,  \TRb_{a-1}$. Put
 $\widetilde \tau:=\tau-\gl\, \TRb_a$. Then $\widetilde\tau$ vanishes
 on $\CL^{a-1}(M)$ and thus in view of the exact sequence Eq.\ \eqref{eq:LeadingSymbol} 
 there is indeed a unique linear functional $T\in\bigl(\cinf{S^*M}\bigr)^*$
 such that $\widetilde \tau = T\circ \sigma_a$.

 For the last statement we note that by Proposition \plref{p:PreHyperTrace} (1.)
 every pretrace is a hypertrace. For the converse note that $\tau=\gl \TRb_a+T\circ\, \sigma_a,  a\in
 \R\setminus \Z_{>-n}$, is indeed a pretrace. For $\TRb_a$ this follows from Proposition
 \ref{p:PreHyperTrace} (4.). For $T\circ\,\sigma_a$ it is a consequence of
 \eqref{eq:SymbolMultiplicative}.
\end{proof} 

The remaining cases of integral values are dealt with in the following:
\begin{theorem}\label{t5} Let $M$ be a closed connected riemannian
 manifold of dimension $n>1$, $a\in \Z_{> -n}$ and
 let $\tau$ be a hypertrace on $\CLa$. Then there are uniquely determined 
 $\gl\in\C$ and a distribution $T\in\bigl(\cinf{S^*M}\bigr)^*$
 such that $\tau=\gl\, \Res_a + T\circ \sigma_a$.

Consequently, a linear functional on $\CLa$ is a hypertrace if and only
 if it is a pretrace.

 \end{theorem}
\begin{proof}
We apply Theorem \ref{t3} and find 
 for $A\in\CL^{a-1}(M)$
\begin{equation}\label{eq:111}  
    A=\sum_{j=1}^N [P_j,Q_j]+\Res(A)\,Q,
\end{equation}
 with $P_j\in\CL^0(M), Q_j\in\CL^{a}(M)$. 
 Thus $\tau(A)=\tau(Q)\Res(A).$ As in the proof of
 Theorem $\ref{t4}$ one now concludes
 $\tau=\tau(Q)\, \Res_a + T\circ \sigma_a$.

 The last statement follows from Proposition
 \ref{p:PreHyperTrace} and the fact that $\Res_a$ and $T\circ\sigma_a$
 are pretraces on $\CLa$.
\end{proof} 
Combining Theorem \ref{t4}, Theorem \ref{t5} and Proposition \ref{p:PreHyperTrace}
we now obtain a complete classification of traces on the algebras $\CLa$, $a\in\Z_{\le 0}$.

\begin{cor}\label{cor:ClassTracesCla} Let $a\in\Z_{\le 0},$ and denote by 
 \[
 \pi_a:\CL^a(M)\longrightarrow \CLa/\CL^{2a-1}(M)
 \]
the quotient map. Let $\tau:\CL^{a}(M)\to\C$ be a trace. 
Then there are uniquely determined $\gl\in\C$ and $T\in \bigl(\CLa/\CL^{2a-1}(M)\bigr)^*$ such that
\begin{equation}
 \tau= \gl\,\TRb_a + T\circ\pi_a.      
\end{equation}
\end{cor}
\begin{remark}\label{r:thetaa}
Note that for $a=1$, the space $\CL^1(M)$ is not an algebra but it is a Lie algebra and it 
makes sense to talk about traces; in this case, the quotient map $\pi_1$ is trivial and 
the proof below shows that $\Res$ is up to normalization the unique trace on $\CL^1(M)$.

In the case $a=0$ this result was known, see \cite{LescPay:UMDEP} 
(and also \cite{Wod:RCHS}).

If $2a\leq-n\leq a$, $\Res_a$ is a non--trivial trace on $\CL^{a}(M)$, however since $\Res\restriction \CL^{2a-1}(M)=0$ (since $2a-1<-n$)
there is $\gl\in \bigl(\CLa/\CL^{2a-1}(M)\bigr)^*, $ such that $\Res_a=\gl\circ\pi_a$.

By choosing right inverses $\theta_a:\cinf{S^*M}\to \CLa$ to the symbol map
one iteratively obtains an isomorphism
\begin{equation}\label{eq:122}  
 \begin{split} 
  \CLa/\CL^{2a-1}(M) &\cong \bigoplus_{k=0}^{|a|} \CL^{a-k}(M)/\CL^{a-k-1}(M) \\
                     &\cong \bigoplus_{k=0}^{|a|} \cinf{S^*M}.
\end{split}
\end{equation}
Under this (non--canonical) isomorphism $T\in 
\bigl(\CLa/\CL^{2a-1}(M)\bigr)^*$ corresponds to a $(|a|+1)$--tuple
$(T_j)_{j=0}^{|a|}$ of distributions $T_j\in \bigl(\cinf{S^*M}\bigr)^*$.
\end{remark}

\begin{proof}
 By Proposition \ref{p:PreHyperTrace}, $\tau_{2a}=\tau\restriction\CL^{2a}(M)$ is a hypertrace on $\CL^{2a}(M)$.
By Theorem \ref{t4} (if $2a<-n+1$) resp.\ Theorem \ref{t5} (if $-n+1\leq2a\leq0$) there is a unique 
$\gl\in\C$ such that 
\begin{equation}\tau_{2a-1}=
\begin{cases}
\gl\,\Tr_{2a-1}, & \text{if }2a<-n+1,\\
\gl\,\Res_{2a-1}, & \text{if }-n+1\leq2a\leq0.
\end{cases}
\end{equation}
Putting
\begin{equation}\widetilde{\tau}=\tau - \gl\,  \TRb_a
\end{equation}
it follows that $\widetilde{\tau}$ vanishes on $\CL^{2a-1}(M)$ and hence is of the form 
$T\circ\pi_a$ for a unique $T\in \bigl(\CLa/\CL^{2a-1}(M)\bigr)^*$.
\end{proof}


\subsection{Alternative approach to Theorem \ref{t5}}        
\label{ss:AAT}

For this subsection we received considerable help from Sylvie Paycha which is
acknowledged with gratitude.

The proof of the uniqueness of the canonical trace $\TR$ (Theorem \ref{t4})
relied solely on the results of Section \ref{s:CohHomDifForms} and Theorem \ref{t:SmoothCommutator}.
The proof of the uniqueness of the residue trace (Theorem \ref{t5}), however, relied additionally
on Theorem \ref{t3} and thus on Proposition \plref{p:SmoothCommutator1} due to Ponge. 
We will give here an alternative completely self--contained
proof of Theorem \ref{t5} which does not make use of Proposition \ref{p:SmoothCommutator1}.

Given a hypertrace $\tau$ on $\CLa$, $a\in \Z, -n<a\le 0$, apply Theorem \ref{t2} with $m=0$.
Then for $A\in\CL^{a-1}(M)$
\begin{equation}
        A=\sum_{j=1}^N [P_j,Q_j] + \Res(A)\, Q+R
\end{equation}
with $P_j\in\CL^0(M), Q_j\in\CLa$ and $R\in\CLsm$. If one can conclude that $\tau(R)=0$ then 
one can proceed as after \eqref{eq:111}. So we have to prove directly

\begin{prop}\label{p:TraceVanishesSmoothing} 
Let $M$ be a closed riemannian manifold and for
$a\in\Z$,  $-n+1\le a\le 0$, let $\tau$ be a hypertrace on $\CLa$. Then $\tau\restriction\CLsm=0$.
\end{prop}
\begin{proof} Let $(U,x_1,\ldots,x_n)$ be a local coordinate chart of $M$. 
Recall that by $\Csym_{\comp}^{a}(U\times\R^n)$ we denote the set of classical symbols of order $a$ on $U$ 
with $U$--compact support, and $\CL_{\comp}^{a}(U)$ denotes the space of
classical pseudodifferential operators of order $a$ on $U$ whose Schwartz kernel has compact support in $U\times U$.
Any operator in $\CL_{\comp}^{a}(U)$ can be extended by zero to an operator in 
$\CL^{a}(M)$, and we have the natural inclusion $\CL_{\comp}^{a}(U)\subset\CL^{a}(M)$.

Note, however, that although for $\sigma\in\CS_{\comp}^a(U\times\R^n)$ the operator $\Op(\sigma)$
maps $\cinfz{U}\to \cinfz{U}$, it does not necessarily lie in $\CL_{\comp}^a(U)$. Below we will take care of this fact by
multiplying by some cut--off function from the right.

Let $\tau\in\sS(\R^n)$ be a Schwartz function with $\int_{\R^n}\tau(\xi)d\xi =1$. By 
Lemma \ref{symbolsumofder} there exist $\tau_1,\ldots,\tau_n\in\Csym^{a}(\R^n)$ such that 
\begin{equation}
\tau=\sum_{k=1}^n\partial_{\xi_k}\tau_k.
\end{equation} 
We note in passing that since the function $\tau$ has non--vanishing integral, at least one of the functions $\tau_k$ 
does not lie in $\sS(\R^n)$.

Next we choose $f\in \cinfz{U}$ with $\int_Uf(x)dx=1$. Then 
$\sigma:=f\otimes\tau$, defined by $\sigma(x,\xi):=f(x)\tau(\xi)$, 
is a smoothing symbol with $U$--compact support. Furthermore,
\begin{align} 
     &\sigma
          =f\otimes\tau=f\otimes\sum_{k=1}^n\partial_{\xi_k}\tau_k=\sum_{k=1}^n\partial_{\xi_k}(f\otimes\tau_k),\\ 
     &\int_{U\times\R^{n}}\sigma(x,\xi)\,d\xi\,dx  = 1.\label{integralsigmanocero}
 \end{align}
 Integration by parts shows that
 (cf.\ \cite[Thm.\ 18.1.6]{Hor:ALPI}, \eqref{eq:108}) 
\begin{equation}\label{Op(sigma)commutators1}
\Op(\sigma)=\sum_{k=1}^n\Op(\partial_{\xi_k}(f\otimes\tau_k))=-i\sum_{k=1}^n[\Op(x_k),\Op(f\otimes\tau_k)].
\end{equation}

 Let $\psi\in \cinfz{U}$ be a function with $\psi=1$ in a neighborhood of
$\supp(f)$; then $\psi f=f$. 
Moreover, for all $k=1,\ldots,n$, 
\begin{equation}
 \begin{split}
[\Op(&x_k),\Op(f\otimes\tau_k)]\Op(\psi)=[\Op(x_k),\Op(f\otimes\tau_k)\Op(\psi)] \\
    &=[\Op(\psi x_k),\Op(f\otimes\tau_k)\Op(\psi)]+A_k, \label{commutatorspsix_k}
\end{split}
\end{equation}
with
\begin{equation}
 \begin{split}
    A_k&:=\Op(f\otimes\tau_k)\Op(\psi)\Op(x_k)\Op(\psi)-\Op(\psi)\Op(f\otimes\tau_k)\Op(\psi)\Op(x_k) \\ 
       &=\Op(f\otimes\tau_k)\Op(\psi)\Op(x_k)(\Op(\psi)-1). \label{operatorAk}
\end{split}
\end{equation}
Here, we used that the operator $\Op(x_k)$ commutes with the operator of multiplication by $\psi$, $\Op(\psi)$, 
cf.\ Remark \ref{r:MultiplicationOp}, and that $\psi f=f$.

Since $f\otimes\tau_k\in\Csym_{\comp}^{a}(U\times\R^n)$, the operator $\Op(f\otimes\tau_k)\Op(\psi)$ 
lies in $\CL_{\comp}^{a}(U)$; similarly, $\psi
x_k\in\Csym_{\comp}^{0}(U\times\R^n)$ and the operator of multiplication 
by $\psi x_k$, $\Op(\psi x_k)$, lies in $\CL_{\comp}^{0}(U)$.

Let $\tau$ be a hypertrace on $\CL^{a}(M)$. 
Then $\tau$ vanishes on  $[\CL_{\comp}^{0}(U),\CL_{\comp}^{a}(U)]$. In particular, for all $k=1,\ldots,n$, 
\[\tau\bigl([\Op(\psi x_k),\Op(f\otimes\tau_k)\Op(\psi)]\bigr)=0.\]
By Proposition \ref{p:PreHyperTrace} (2.),  we have $\tau\restriction \CLsm=\gl \Tr$ for some $\gl\in\C$.
Now, since $\psi=1$ near the support of $f$, by \eqref{operatorAk} the operator $A_k$ is smoothing and its 
Schwartz kernel vanishes on the diagonal. Hence,  its $L^2$--trace vanishes
and thus also $\tau(A_k)=\gl \Tr(A_k)=0$. 

Thus, for $\Op(\sigma)\Op(\psi)\in\CL_{\comp}^{-\infty}(U)$, from \eqref{Op(sigma)commutators1} and 
\eqref{commutatorspsix_k} we conclude  
\begin{equation}\label{LambdainOpsigmacommutators}
  \begin{split}
 \tau(&\Op(\sigma)\Op(\psi))
     =-i\sum_{k=1}^n\tau\bigl([\Op(x_k),\Op(f\otimes\tau_k)]\Op(\psi)\bigr) \\
   &=-i\sum_{k=1}^n \left(\tau\bigl([\Op(\psi x_k),\Op(f\otimes\tau_k)\Op(\psi)]\bigr)+\tau(A_k)\right) =0.
  \end{split}
\end{equation}
On the other hand, by \eqref{defL2tracesymbol} and Proposition
\ref{p:PreHyperTrace} (2.), 
\begin{equation}
 \tau(\Op(\sigma)\Op(\psi))=\gl\,\Tr(\Op(\sigma)\Op(\psi))=\gl\,\int_{U\times\R^{n}}\sigma(x,\xi)\,d\xi\,dx
 =\gl.
\end{equation} 
Therefore, by \eqref{LambdainOpsigmacommutators} we obtain $\gl=0$.
\end{proof}

\section{Extension to vector bundles}\label{s:TracesVectBund}     

In this final section we extend the classification of traces and hypertraces
to the spaces $\CLaE$ of pseudodifferential operators acting on sections
of the vector bundle $E$ over $M$.

\subsection{Preliminaries} \label{ss:ExtVectBundPrelim}
Unless otherwise said, in the whole section $M$ denotes a smooth closed connected 
riemannian manifold of dimension $n$. Let $E\to M$ be  a smooth hermitian vector bundle
over $M$. We denote by $\CLaE$ the space of classical pseudodifferential 
operators of order $a$ acting on the sections of $E$. $\CLaE$ acts naturally
as (unbounded) operators on the Hilbert space $L^2(M,E)$ of square integrable
sections of $E$. The elementary discussion of 
traces, pretraces and hypertraces in Subsection \ref{ss:Traces} extends
verbatim to $\CLaE$. However, as noted there, we now 
only have $[\CLaE,\CL^{b}(M,E)]\subset\CL^{a+b}(M,E)$ as opposed to 
$[\CLa,\CL^{b}(M)]\subset\CL^{a+b-1}(M)$ in the scalar case $E=M\times\C$.
Lemma \plref{l:commCl02aaa} holds with the same
proof for $\CL^\bullet(M,E)$ instead of $\CL^\bullet(M)$. Finally, 
Theorem \ref{t:SmoothCommutator} holds for $\CLsmE$ too; this follows 
directly from \cite[Thm.\ A.1]{Gui:RTFIO}, which is stated in a Hilbert space
context and is therefore flexible enough. 

In sum, also Proposition \plref{p:PreHyperTrace} (1.--3.) holds accordingly:

\begin{prop}\label{p:PreHyperTraceE} Let $a\in\R$.
 
 \textup{1. }
 Any pretrace on $\CLaE$ is a hypertrace on $\CLaE$.
 
 \textup{2.}
 If $\tau$ is a hypertrace on $\CL^a(M,E)$ then there is a unique constant $\lambda\in\C$
 such that $\tau\restriction\CLsmE=\lambda\, \Tr$.

 \textup{3.}
 If $a\in\Z_{\le 0}$ and $\tau$ is a trace on $\CL^{a}(M,E)$ then
 $\tau\restriction \CL^{2a}(M,E)$ is a pretrace (and hence a hypertrace). 
 Conversely, given a pretrace on $\CL^{2a}(M,E)$ then any linear extension 
 $\widetilde \tau$ of $\tau$ to $\CL^{a}(M,E)$ is a trace.
 
 \end{prop}
 
 For the analogue of Proposition \plref{p:PreHyperTrace} (4.) see Proposition
 \ref{p:HypertracesE}. 
 
The main task now is to classify the hypertraces on $\CLaE$.

\subsection{Trivial vector bundles}\label{ss:TrivialVectBund}

Let $M_N(\C)$ be the space of $N\times N$ matrices with coefficients in $\C$. For all $i,j=1,\ldots,N$, 
we denote by $E_{ij}$ the elementary matrix in $M_N(\C)$ with $1$ in the $(i,j)$--position and 0 everywhere else. 
The matrices $E_{ij}$ form a basis of $M_N(\C)$ and we have
\begin{equation}
 \begin{split}\label{eq:MatrixRelations}
        E_{ij} \, E_{kl} \, = \, \delta_{jk}\, E_{il}.
 \end{split}
\end{equation}
Let us denote by $\tr_N$ the unique trace on the algebra $M_N(\C)$ such that for all $i=1,\ldots,N$, $\tr_N(E_{ii})=1$.

For a complex vector space $V$ we will tacitly identify $M_N(V)$ with
$V\otimes M_N(\C)$ via 
\begin{equation} 
x:=\left(x_{ij}\right)_{i,j}\longmapsto  \sum_{i,j=1}^N x_{ij}\otimes E_{ij}.
\end{equation}
Obviously, we have $\CL^a(M,\C^N)=M_N\bigl(\CLa\bigr)\cong \CLa\otimes
M_N(\C)$. Here, by slight abuse of notation $\CL^a(M,\C^N)$ denotes the
space of classical pseudodifferential operators acting on the trivial
vector bundle $M\times \C^N$.

\begin{dfn}\label{def:TauTimesTr} Let $a\in\R$ and let $\tau$ be a linear
 functional on $\CLa$. Then we put
\begin{equation}\label{eq:112}   
  \begin{split}
      \tau\otimes \tr_N:\CL^a(M,\C^N)&\longrightarrow \C,\\
              A:=\left(A_{ij}\right)_{i,j}&\mapsto \sum_{i,j=1}^N
              (\tau\otimes\tr_N) (A_{ij}\otimes E_{ij})=\sum_{i=1}^N \tau(A_{ii}).
  \end{split}
\end{equation}
It is straightforward to check that if $\tau$ is a hypertrace (pretrace,
trace) on $\CLa$ then $\tau\otimes\tr_N$ is a hypertrace (pretrace, trace)
on $\CL^a(M,\C^N)$.
\end{dfn}

\begin{prop}\label{p:HypertracesCN}
 Let $a\in\R$. Then every hypertrace on $\CL^a(M,\C^N)$ is of the
 form $\tau\otimes\tr_N$ with a unique hypertrace $\tau$ on $\CLa$.
\end{prop}
\begin{proof} Let $T$ be a hypertrace on $\CL^a(M,\C^N)$. For $i,j=1,\ldots,N$
 we put $T_{ij}:\CLa\to\C$, $T_{ij}(A):=T(A\otimes E_{ij})$.

Since $\Id\in\CL^0(M,\C^N)$ we infer from the hypertrace property
\begin{equation}\label{eq:116}  
 \begin{split}
        T_{ij}(A)&=T(A\otimes E_{ij})=
           T\bigl(  (A\otimes E_{i1}) \, (\id \otimes E_{1j}) \bigr )\\
                 &=T \bigl ( (\id\otimes E_{1j})\, (A\otimes E_{i1}) \bigr) =
                 \delta_{ij}\, T_{11}(A),
 \end{split}                
\end{equation}
thus $T_{ij}=0$ for $i\not =j$ and $T_{11}=T_{22}=\ldots=T_{NN}=:\tau$.

$\tau$ is a hypertrace on $\CLa$. Namely, for $A\in\CLa, B\in\CL^0(M)$
we have
\begin{equation}\label{eq:117}   
  \begin{split}
   \tau(AB) &= T((AB)\otimes E_{11})
     =T\bigl( (A\otimes E_{11})\, (B\otimes E_{11} ) \bigr) \\
            &= T\bigl( (B\otimes E_{11})\, (A\otimes E_{11}) \bigr) = \tau(BA).
  \end{split}
\end{equation}
Certainly, we have $T=\tau\otimes \tr_N$. 

For the uniqueness we only have to note that if  $T=\tau\otimes \tr_N$ 
then $\tau(A)=T(A\otimes E_{11})$.
\end{proof} 

\subsection{General vector bundles}\label{ss:Generalvb}

Let $E$ be a vector bundle over $M$. By Swan's Theorem there is a positive integer $N$, 
such that $E$ is a direct summand of $M\times\C^N$; let $e\in
M_N(\cinf{M})=\cinf{M,M_N(\C)}$ be a smooth projection onto $E$.
Then the $\cinf{M}$--module of smooth sections of $E$ is given by 
\begin{equation}\label{eq:137}  
\Gamma^\infty(M,E)\cong e(\cinf{M}^N).
\end{equation}
Note that since we assumed $M$ to be connected (cf.\ Subsection
\ref{ss:ExtVectBundPrelim}), the idempotent valued function $e$ has constant rank.

The following lemma is well--known. Since we could not find a place where
it is stated as needed we provide, for convenience, a quick proof:

\begin{lemma}\label{l:ReR}
 Let $\cA:=\cinf{M,M_N(\C)}$. Then $\cA\, e\,\cA=\cA$.
Equivalently there exist $p_j,q_j\in \cinf{M,M_N(\C)}$, $j=1,\ldots,r$
   such that 
   \begin{equation}\label{eq:PartUnity}
        \sum_{j=1}^r p_j\,  e\, q_j = 1_M\otimes I_N,   
    \end{equation}
where $1_M$ denotes the function which is constant $1$ on $M$ and $I_N$ is the
$N\times N$ identity matrix.
\end{lemma}
\begin{proof} It obviously suffices to prove Eq.\ \eqref{eq:PartUnity}.
Choose a finite partition of unity $\psi_j, j=1,\ldots,s$, smooth functions
$\chi_j\in\cinf{M}$ such that $\chi_j=1$ in a neighborhood of $\supp(\psi_j)$
and such that in a neighborhood $U_j$ of $\supp(\chi_j)$ there 
is a smooth map $v:U_j\to M_N(\C)$ such that
\[
     v\, e \,  v\ii = e_k := \begin{pmatrix} I_k & 0 \\ 0 & 0 \end{pmatrix}.
\]      
Choose $N\times N$ matrices $a_l,b_l, l=1,\ldots,t$,  with
\[
    \sum_{l=1}^t a_l\, e_k\, b_l = I_N.
\]
We tacitly view $a_l, b_l$ also as constant matrix valued functions on $M$.
Slightly abusing notation we now find the decomposition
\[
  \begin{split}
   1_M \otimes I_N & = \sum_{j=1}^s \psi_j\chi_j \otimes I_N
                         = \sum_{j=1}^s\sum_{l=1}^t \psi_j\, v\ii\, a_l\, e_k\,
                         b_l\, v \,  \chi_j \\
                   & = \sum_{j=1}^s\sum_{l=1}^t \bigl(\psi_j\, v\ii\, a_l\,
                   v\bigr )\, e \,  
                   \bigl(  v\ii\, b_l\, v\, \chi_j).\qedhere
  \end{split}
\]
\end{proof}

For a linear functional $\tau$ on $\CL^a(M,\C^N)$ we now put 
\begin{equation}\label{eq:136}  
   \tau_E(A):= \tau(eAe). 
\end{equation}
This definition depends on the choice of the idempotent $e$ and is
therefore not canonical. As in the scalar case if $\CLaE\subset
\CL^b(M,E)$ we write $\tau_{E,a}:=\tau_E\restriction \CL^a(M,E)$.

The canonical trace $\TR$ and the residue trace $\Res$ are naturally
defined on $\CL^\bullet(M,E)$ for any vector bundle $E$ (cf.\ \cite{Les:NRP}). 
To distinguish them let us for the moment denote by $\TR^{(N)}, \Res^{(N)}$
the corresponding functionals on $\CL^\bullet(M,\C^N)$ and
by $\TR^{(E)}, \Res^{(E)}$ the corresponding functionals on
$\CL^\bullet(M,E)$.

Then one immediately checks that
\begin{align}
 \TR^{(N)}  &=\TR\otimes \tr_N,  &\TR^{(E)}  &=\bigl(\TR\otimes \tr_N)_E,\\
 \Res^{(N)} &=\Res\otimes \tr_N , & \Res^{(E)} &=\bigl(\Res\otimes \tr_N)_E,
\end{align}
hence $\TR$ and $\Res$ are compatible with the operations 
$\tau\mapsto \tau \otimes\tr_N$ and $\tau\mapsto \tau_E$ in the most natural way.

From now on we will write $\TR_E$ for $\TR^{(E)}$, and $\Res_E$ for $\Res^{(E)}$.
A confusion with the notation introduced in Definition \plref{def:PreHyperTrace}
should not arise.

We also extend the linear functional $\Trt$ of Definition 
\ref{def:TRRes} to $\CL^0(M,E)$ by defining 
\[\Trt_E:=\bigl(\Trt\otimes\tr_N\bigr)_E.\]
Since $\Trt$ is not a trace, this definition may depend on the choice 
of the idempotent $e$, hence is not canonical; but $\Trt$ already depended
on a choice.

Finally we put $\TRb_{E,a}:=\bigl(\TRb_a\otimes \tr_N\bigr)_E$ on
$\CLaE$. From Subsection \ref{s:ClassTrac} we see
\begin{equation}\label{eq:136a}  
 \TRb_{E,a}:= 
\begin{cases}
 \TR_{E,a}, & \textup{if }a\in\R\setminus \Z_{\ge -n},\\
 \Trt_{E,a}, & \textup{if } a\in \Z, -n\le a < \frac{-n+1}{2},\\
 \Res_{E,a}, & \textup{if } a\in \Z, \frac{-n+1}{2}\le a.
\end{cases}       
\end{equation}

\begin{prop}\label{p:HypertracesE} 
 
\textup{1.} Let $a\in\R$
and let $\tau$ be a hypertrace (resp. pretrace, trace) on $\CL^a(M,\C^N)$. 
Then $\tau_E: \CLaE\longrightarrow \C, \quad A\mapsto \tau (e A e)$  
is a hypertrace (resp. pretrace, trace) on $\CLaE$.

\textup{2.} Any hypertrace on $\CLaE$ is of the form $\bigl(\tau\otimes\tr_N)_E$
for a unique hypertrace $\tau$ on $\CLa$.

\textup{3.} For $a\in\Z_{\le 0}$, $\TRb_{E,a}$ is a trace on $\CLaE$.
For $a\in \R\setminus\bigl(\Z\cap [-n+1,-n/2]\bigr)$ it is a pretrace (and hence a hypertrace).
\end{prop}
\begin{proof}
 \textup{1.} To prove that the linear functional $\tau_E$ is a hypertrace consider
$A\in \CL^a(M,E), B\in\CL^0(M,E)$. Then
\begin{equation}\label{eq:118}   
  \begin{split}
       \tau_E(AB)&=\tau(eABe)=\tau\bigl((eAe)(eBe)\bigr)\\
                 &=\tau\bigl((eBe)(eAe)\bigr)=\tau(eBAe)=\tau_E(BA).
  \end{split}
\end{equation}
Note that $eAe\in\CL^a(M,\C^N), eBe\in\CL^0(M,\C^N)$. Repeating the argument
with $A,B\in \CL^{a/2}(M,E)$ shows that if $\tau$ is a pretrace then so
is $\tau_E$. Similarly if $a\in\Z_{\le 0}$ and
$\tau$ is a trace then $\tau_E$ is a trace. 

\textup{2.}
Conversely, let $T$ be a hypertrace on $\CLaE$. 
We choose $p_j, q_j, j=1,\ldots,r$ according to Lemma \ref{l:ReR}.
We will repeatedly use that multiplication by $p_j, q_j$ is in
$\CL^0(M,\C^N)$,
resp. that multiplication by $ep_je, eq_je$ is in $\CL^0(M,E)$.

Suppose we had a hypertrace $\tilde T$ on $\CL^a(M,\C^N)$ such
that $\tilde T_E=T$. Then for $A\in\CL^a(M,\C^N)$
\begin{equation}\label{eq:119}  
 \begin{split}
  \tilde T(A)&=  \tilde T\bigl( (1_M\otimes I_N) A\bigr) =\sum_{j=1}^r 
  \tilde T\bigl( p_j e q_j A \bigr) =\sum_{j=1}^r \tilde T\bigl( p_j e^2 q_j A \bigr)  \\
      &= \sum_{j=1}^r \tilde T\bigl( e q_j A p_j e \bigr) =\sum_{j=1}^r \tilde T\bigl( e^2 q_j A p_j e^2 \bigr)= \sum_{j=1}^r T\bigl( e q_j A p_j e).
   \end{split}
\end{equation}
Thus there is at most such a $\tilde T$. We now \emph{define} $\tilde T$
by the right hand side of Eq.\ \eqref{eq:119}.
We have $\tilde T_E=T$. Indeed, for $A\in\CLaE$
\begin{equation}\label{eq:120}  
 \begin{split}
  \tilde T_E(A) &= \tilde T(e A e ) = 
      \sum_{j=1}^r T\bigl( (e\, q_j\, e\, A\, e)\, ( e\, p_j\, e) \bigr)\\
                &= \sum_{j=1}^r T\bigl( e\, p_j\, e\, q_j\, e\, A\, e \bigr) =
                T(e A e)= T(A).
               \end{split}                
\end{equation}
In the last line we used  Eq.\ \eqref{eq:PartUnity}.

Next we show that $\tilde T$ is a hypertrace on $\CL^a(M,\C^N)$. Indeed,
for $A\in\CL^a(M,\C^N), B\in\CL^0(M,\C^N)$ we find using Eq.\
\eqref{eq:PartUnity},
\begin{equation}\label{eq:121}   
  \begin{split}
   \tilde T(AB)&= \sum_{j=1}^r T\bigl( e\, q_j A\, (1_M\otimes I_N)\, B\,
   p_j\, e\bigr)
                         = \sum_{j,k=1}^r T\bigl( e\, q_j\, A\, p_k\, e\,
                         q_k\, B\, p_j\, e\bigr)\\
               &= \sum_{j,k=1}^r T\bigl( e\, q_k\, B\, p_j\, e\, q_j\, A\,
               p_k\, e\bigr)
               = \sum_{k=1}^r T\bigl( e\, q_k\, B\,  A\, p_k\, e\bigr) =\tilde T(B A).
  \end{split}
\end{equation}
By Proposition \ref{p:HypertracesCN} there is now a unique hypertrace
$\tau$ on $\CLa$ such that $\tilde T= \tau\otimes\tr_N$. Then we
conclude $T=\widetilde T_E=\bigl(\tau\otimes \tr_N)_E$. 
Recall that $\tilde T$ is uniquely determined by $T$ and
$\tau$ is uniquely determined by $\tilde T$, whence $\tau$
is uniquely determined by $T$.

\textup{3.} follows from the proved part 1., Eq.\ \eqref{eq:136a} and
Proposition \plref{p:PreHyperTrace}.
\end{proof}

Before stating the final result, we have to clarify what leading symbol traces on 
$\CLaE$ look like.
For the moment consider a closed manifold $X$ with a vector bundle $E\to X$.
We can construct traces on the noncommutative algebra $\Gamma^\infty(X,\End\, E)$
as follows: first the fiberwise trace induces a linear map
\begin{equation}
 \begin{split}
 \tr_E:\Gamma^\infty(X,\End\, E)\to\cinf{X}\\
 \tr_E(s)(x)=\tr_{E_x}\big(s(x)\big).
 \end{split}
\end{equation}
$\tr_E$ vanishes on commutators. Thus for any $T\in\bigl(\cinf{X}\bigr)^*$ the composition
$T\circ\tr_E$ is a trace on $\Gamma^\infty(X,\End\, E)$.

It is straightforward to see that indeed all traces on $\Gamma^\infty(X,\End\, E)$ are of 
this form. Since we will not use this fact we leave the details of proof to the reader: 

\begin{prop} Let $X$ be a closed manifold and let $E$ be a vector bundle over $X$.
Then for any trace $\tau$ on $\Gamma^\infty(X,\End\, E)$ there is a unique distribution 
$T\in\bigl(\cinf{X}\bigr)^*$ such that $\tau=T\circ\tr_E$.\end{prop}

The final result is now a consequence of Theorems \ref{t4}, \ref{t5},
Corollary \ref{cor:ClassTracesCla}, and Propositions \ref{p:PreHyperTraceE},
\ref{p:HypertracesE}.

\begin{theorem}\label{t6} Let $M$ be a closed connected riemannian manifold
 of dimension $n>1$ and let $E$ be a complex vector bundle over $M$. Denote
 by $\Pi:E\to M$ the projection map, by $\sigma_a:\CLaE\to
 \Gamma^\infty\bigl(S^*M,\Pi^*\End\, E\bigr)$ the leading symbol map, and by
 $\tr_E$ the fiberwise trace
 $\Gamma^\infty\bigl(S^*M,\Pi^*\End\, E\bigr)\to\cinf{S^*M}$. 
 Fix $N$ and an idempotent $e$ as in \textup{Eq.}\ \eqref{eq:137} and let $\TRb_{E,a}$
 be as defined in \textup{Eq.}\ \eqref{eq:136a}.
 
 \textup{1. } Let $a\in\R$ and let $\tau$ be a hypertrace
 on $\CLaE$. Then there are uniquely determined $\gl\in\C$ and a distribution
 $T\in\bigl(\cinf{S^*M}\bigr)^*$
 such that 
\begin{equation}\label{eq:140}  
 \tau=T \circ \tr_E\circ\, \sigma_a +
     \begin{cases}
        \gl\, \TRb_{E,a},  & \textup{if } a\notin\Z_{>-n}, \\
        \gl\, \Res_{E,a},  & \textup{if } a\in\Z_{>-n}.
     \end{cases}
\end{equation}
 
 \textup{2. } Let $a\in \Z_{\leq0}$ and denote by 
 \[\pi_a:\CLaE\to\CLaE/\CL^{2a}(M,E)\] the quotient map. Furthermore, let 
 \[\theta_a:\CLaE/\CL^{2a}(M,E)\to\CLaE\] be a right inverse to $\pi_a$.

Let $\tau:\CLaE\to\C$ be a trace.
Then there are uniquely determined $\gl\in\C$, $T\in\bigl(\cinf{S^*M}\bigr)^*$ and
$\Phi\in\left(\CLaE/\CL^{2a}(M,E)\right)^*$ such that 
 \begin{equation} 
 \tau=\gl\,
 \TRb_{E,a}+T\circ\tr_E\circ\,\sigma_{2a}(\id-\theta_a\circ\pi_a)+\Phi\circ\pi_a.
\end{equation}
\end{theorem}
In the first line of Eq.\ \eqref{eq:140} the case $a=-n$ is included, thus 
we write $\TRb_{E,a}$ instead of $\TR_{E,a}$ there.
\begin{proof}
The right inverse $\theta_a$ can be constructed successively from the map $\Op$, 
cf.\ Remark \ref{r:thetaa}. 

 \textup{1. } By Proposition \ref{p:HypertracesE} there is a unique hypertrace 
 $\widetilde{\tau}$ on $\CLa$ such that $\tau=\left(\widetilde{\tau}\otimes\tr_N\right)_E$.
 The claim now follows from Theorem \ref{t4} and Theorem \ref{t5} applied to
 $\widetilde{\tau}$. Note that 
 $T\circ \tr_E \circ\, \sigma_a = \bigl((T\circ \sigma_a ) \otimes
 \tr_N\bigr)_E$, cf.\ Eq.\ \eqref{eq:136} and Definition \ref{def:TauTimesTr}.

 \textup{2. } Let $a\in \Z_{\le 0}$. By Proposition \ref{p:PreHyperTraceE}, $\tau\restriction\CL^{2a}(M,E)$ is
 a hypertrace. Thus, by the proved part 1.\ we have 
\begin{equation}\label{eq:138}  
 \tau\restriction\CL^{2a}(M,E)=T\circ\tr_E\circ\, \sigma_{2a}+
 \begin{cases}
    \gl\,\TRb_{E,2a},  & \textup{if } 2a\le -n, \\
     \gl\,\Res_{E,2a},  & \textup{if } 2a>-n.
 \end{cases}
\end{equation}
We emphasize that by Eq.\ \eqref{eq:136a} 
\begin{equation}\label{eq:139}  
 \TRb_{E,a}\restriction \CL^{2a}(M,E)  =   
     \begin{cases}
         \gl\,\TRb_{E,2a},  & \textup{if } 2a\le -n, \\
         \gl\,\Res_{E,2a},  & \textup{if } 2a>-n.
     \end{cases}
\end{equation}
Consider for $A\in\CLaE$ 
\[
\widetilde{\tau}(A) := \tau(A)- \gl\, \TRb_{E,a}(A)-
T\circ\tr_E\circ\,\sigma_{2a}(A-\theta_a\circ\pi_a(A)).
\]
Then due to Eq.\ \eqref{eq:139} and Eq.\ \eqref{eq:138} the functional
$\widetilde{\tau}$ vanishes on $\CL^{2a}(M,E)$ and thus is of the form 
$\Phi\circ\pi_a$ with $\Phi\in\left(\CLaE/\CL^{2a}(M,E)\right)^*$. Then
$\tau=\Phi\circ\pi_a+\widetilde{\tau}$ and the theorem is proved.
\end{proof}



\begin{thebibliography}{\textsc{Wod87b}}

\bibitem[\textsc{BoTu82}]{BotTu:DFA}
\textsc{R.~Bott} and \textsc{L.~W. Tu}, \emph{Differential forms in algebraic
  topology}, Graduate Texts in Mathematics, vol.~82, Springer-Verlag, New York,
  1982. \MR{658304 (83i:57016)}

\bibitem[\textsc{BrGe87}]{BryGet:HAPDO}
\textsc{J.-L. Brylinski} and \textsc{E.~Getzler}, \emph{The homology of
  algebras of pseudodifferential symbols and the noncommutative residue},
  $K$-Theory \textbf{1} (1987), no.~4, 385--403. \MR{920951 (89j:58135)},
  \texttt{http://dx.doi.org/10.1007/BF00539624}

\bibitem[\textsc{CdS01}]{CdS:SG}
\textsc{A.~Cannas~da Silva}, \emph{Lectures on symplectic geometry}, Lecture
  Notes in Mathematics, vol. 1764, Springer-Verlag, Berlin, 2001. \MR{1853077
  (2002i:53105)}, \texttt{http://dx.doi.org/10.1007/978-3-540-45330-7}

\bibitem[\textsc{FGLS96}]{FGLS:NRM}
\textsc{B.~V. Fedosov}, \textsc{F.~Golse}, \textsc{E.~Leichtnam}, and
  \textsc{E.~Schrohe}, \emph{The noncommutative residue for manifolds with
  boundary}, J. Funct. Anal. \textbf{142} (1996), no.~1, 1--31. \MR{1419415
  (97h:58157)}, \texttt{http://dx.doi.org/10.1006/jfan.1996.0142}

\bibitem[\textsc{GrSe95}]{GruSee:WPP}
\textsc{G.~Grubb} and \textsc{R.~T. Seeley}, \emph{Weakly parametric
  pseudodifferential operators and {A}tiyah-{P}atodi-{S}inger boundary
  problems}, Invent. Math. \textbf{121} (1995), no.~3, 481--529. \MR{1353307
  (96k:58216)}

\bibitem[\textsc{Gui85}]{Gui:NPW}
\textsc{V.~Guillemin}, \emph{A new proof of {W}eyl's formula on the asymptotic
  distribution of eigenvalues}, Adv. in Math. \textbf{55} (1985), no.~2,
  131--160. \MR{772612 (86i:58135)}

\bibitem[\textsc{Gui93}]{Gui:RTFIO}
\textsc{V.~Guillemin}, \emph{Residue traces for certain algebras of {F}ourier
  integral operators}, J. Funct. Anal. \textbf{115} (1993), no.~2, 391--417.
  \MR{1234397 (95a:58123)},\\ \texttt{http://dx.doi.org/10.1006/jfan.1993.1096}

\bibitem[\textsc{H{\"o}r71}]{Hor:FIOI}
\textsc{L.~H{\"o}rmander}, \emph{Fourier integral operators. {I}}, Acta Math.
  \textbf{127} (1971), no.~1-2, 79--183. \MR{0388463 (52 \#9299)}

\bibitem[\textsc{H{\"o}r03}]{Hor:ALPI}
\bysame, \emph{The analysis of linear partial differential operators. {I}},
  Classics in Mathematics, Springer-Verlag, Berlin, 2003, Distribution theory
  and Fourier analysis, Reprint of the second (1990) edition [Springer, Berlin;
  MR1065993 (91m:35001a)]. \MR{1996773}

\bibitem[\textsc{KoVi95}]{KonVis:GDE}
\textsc{M.~Kontsevich} and \textsc{S.~Vishik}, \emph{Geometry of determinants
  of elliptic operators}, Functional analysis on the eve of the 21st century,
  {V}ol.\ 1 ({N}ew {B}runswick, {NJ}, 1993), Progr. Math., vol. 131,
  Birkh\"auser Boston, Boston, MA, 1995, pp.~173--197. \texttt{arXiv:9406140
  [hep-th]}, \MR{1373003 (96m:58264)}

\bibitem[\textsc{LePa07}]{LescPay:UMDEP}
\textsc{J.-M. Lescure} and \textsc{S.~Paycha}, \emph{Uniqueness of
  multiplicative determinants on elliptic pseudodifferential operators}, Proc.
  Lond. Math. Soc. (3) \textbf{94} (2007), no.~3, 772--812. \MR{2325320
  (2009a:58043)}, \texttt{http://dx.doi.org/10.1112/plms/pdm004}

\bibitem[\textsc{Les99}]{Les:NRP}
\textsc{M.~Lesch}, \emph{On the noncommutative residue for pseudodifferential
  operators with log-polyhomogeneous symbols}, Ann. Global Anal. Geom.
  \textbf{17} (1999), no.~2, 151--187. \texttt{arXiv:9708010 [dg-ga]},
  \MR{1675408 (2000b:58050)}

\bibitem[\textsc{Les10}]{Les:PDO}
\bysame, \emph{Pseudodifferential operators and regularized traces}, Motives,
  quantum field theory, and pseudodifferential operators, Clay Math. Proc.,
  vol.~12, Amer. Math. Soc., Providence, RI, 2010, pp.~37--72.
  \texttt{arXiv:0901.1689 [math.OA]}, \MR{2762524}

\bibitem[\textsc{MSS08}]{MSS:UKVT}
\textsc{L.~Maniccia}, \textsc{E.~Schrohe}, and \textsc{J.~Seiler},
  \emph{Uniqueness of the {K}ontsevich-{V}ishik trace}, Proc. Amer. Math. Soc.
  \textbf{136} (2008), no.~2, 747--752 (electronic). \texttt{math/0702250v1
  [math.FA] 9 Feb 2007}, \MR{2358517 (2008k:58062)},\\
  \texttt{http://dx.doi.org/10.1090/S0002-9939-07-09168-X}

\bibitem[\textsc{NJ10}]{Nei:CCS}
\textsc{C.~Neira~Jim{\'e}nez}, \emph{Cohomology of classes of symbols and
  classification of traces on corresponding classes of operators with non
  positive order}, Doctoral dissertation, Universit\"at Bonn, 2010.
  \texttt{http://hss.ulb.uni-bonn.de/2010/2214/2214.htm \textrm{resp.} .pdf}

\bibitem[\textsc{PaRo04}]{PayRos:TCCLS}
\textsc{S.~Paycha} and \textsc{S.~Rosenberg}, \emph{Traces and characteristic
  classes on loop spaces}, Infinite dimensional groups and manifolds, IRMA
  Lect. Math. Theor. Phys., vol.~5, de Gruyter, Berlin, 2004, pp.~185--212.
  \MR{2104357 (2005h:58047)}

\bibitem[\textsc{Pay07}]{Pay:NRC}
\textsc{S.~Paycha}, \emph{The noncommutative residue and canonical trace in the
  light of {S}tokes' and continuity properties},  \texttt{arXiv:0706.2552
  [math.OA]}.

\bibitem[\textsc{Pon10}]{Pon:TPO}
\textsc{R.~Ponge}, \emph{Traces on pseudodifferential operators and sums of
  commutators}, J. Anal. Math. \textbf{110} (2010), 1--30.
  \texttt{arXiv:0707.4265 [math.AP]}

\bibitem[\textsc{Sco10}]{Sco:TDP}
\textsc{S.~Scott}, \emph{Traces and determinants of pseudodifferential
  operators}, Oxford Mathematical Monographs, Oxford University Press, Oxford,
  2010. \MR{2683288}

\bibitem[\textsc{Shu01}]{Shu:POST}
\textsc{M.~A. Shubin}, \emph{Pseudodifferential operators and spectral theory},
  second ed., Springer-Verlag, Berlin, 2001, Translated from the 1978 Russian
  original by Stig I. Andersson. \MR{1852334 (2002d:47073)}

\bibitem[\textsc{Wod84}]{Wod:LISA}
\textsc{M.~Wodzicki}, \emph{Local invariants of spectral asymmetry}, Invent.
  Math. \textbf{75} (1984), no.~1, 143--177. \MR{728144 (85g:58089)},
  \texttt{http://dx.doi.org/10.1007/BF01403095}

\bibitem[\textsc{Wod87a}]{Wod:RCHS}
\textsc{M.~Wodzicki}, \emph{Report on the cyclic homology of symbols}, Jan
  1987, manuscript, IAS Princeton. Available online at
  http://math.berkeley.edu/$\sim$wodzicki.

\bibitem[\textsc{Wod87b}]{Wod:NRF}
\textsc{M.~Wodzicki}, \emph{Noncommutative residue. {I}. {F}undamentals},
  {$K$}-theory, arithmetic and geometry ({M}oscow, 1984--1986), Lecture Notes
  in Math., vol. 1289, Springer, Berlin, 1987, pp.~320--399. \MR{923140
  (90a:58175)}

\end{thebibliography}

\providecommand{\bysame}{\leavevmode\hbox to3em{\hrulefill}\thinspace}
\providecommand{\MR}{\relax\ifhmode\unskip\space\fi MR }
\providecommand{\MRhref}[2]{%
  \href{http://www.ams.org/mathscinet-getitem?mr=#1}{#2}
}
\providecommand{\href}[2]{#2}

\end{document}